\documentclass[a4paper, 12pt]{amsart}
\usepackage[T1]{fontenc}
\usepackage{xspace}
\usepackage{enumerate}
\usepackage[utf8]{inputenc} 
\usepackage[english]{babel}
\usepackage{lmodern}
\usepackage{microtype}
\usepackage{amssymb}
\usepackage{verbatim}
\usepackage{amsthm}
\usepackage{amsmath}
\usepackage{mathtools}
\usepackage{color}
\usepackage{tikz}
\usepackage{threeparttable}
\usepackage{graphicx}
\usepackage{caption}
\usepackage{paralist}
\usepackage[pagebackref=true, pdfborder={0 0 0}, linktocpage]{hyperref}
\usepackage{backref}
\usepackage{aliascnt}
\usepackage{cleveref}
\usepackage{dsfont}
\usepackage{mathrsfs}
\usepackage{bm}

\newtheorem{theorem}{Theorem}[section]
\newtheorem{lemma}[theorem]{Lemma}

\theoremstyle{remark}

\numberwithin{equation}{section}

\begin{document}

\title[Continuous time limit of EnKFs]{Derivation of Ensemble Kalman-Bucy Filters with unbounded nonlinear coefficients}
\author{Theresa Lange}
\address{Fakult\"at f\"ur Mathematik, Universit\"at Bielefeld, D-33501 Bielefeld, Germany}
\email{tlange@math.uni-bielefeld.de}

\date{November 26, 2021}

\begin{abstract}
We provide a rigorous derivation of the Ensemble Kalman-Bucy Filter as well as the Ensemble Transform Kalman-Bucy Filter in case of nonlinear, unbounded model and observation operators. We identify them as the continuous time limit of the discrete-time Ensemble Kalman Filter and the Ensemble Square Root Filters, respectively, together with concrete convergence rates in terms of the discretization step size. Simultaneously, we establish well-posedness as well as accuracy of both the continuous-time and the discrete-time filtering algorithms.
\end{abstract}


\maketitle 
\section{Introduction}
\noindent
Continuous-time stochastic filtering is concerned with identifying the conditional distribution of a continuous-time system of interest given noisy observations. Precisely, consider the following setting:
\begin{align}
{\rm d}X_t &= f\left(X_t\right){\rm d}t + Q^{\frac{1}{2}}{\rm d}W_t, \hspace{0.25cm} X_t \in \mathds{R}^{d},\label{X}\\
{\rm d}Y_t &= g\left(X_t\right){\rm d}t + C^{\frac{1}{2}}{\rm d}V_t, \hspace{0.25cm} Y_t \in \mathds{R}^{p},\label{Y}
\end{align}
where $Q \in \mathds{R}^{d \times d}$ and $C \in \mathds{R}^{p \times p}$ are symmetric positive definite matrices and $W$ and $V$ are independent Brownian motions. In this case, the conditional distribution
\[ \pi_t({\rm d}x) := \mathds{P}\left[X_t \in {\rm d}x | \mathcal{Y}_{0:t}\right]\]
where $\mathcal{Y}_{0:t} := \sigma\left(Y_s, s\leq t\right)$, is given by the Kushner-Stratonovich equation
\begin{equation*}
{\rm d}\pi_t(\varphi) = \pi_t(\mathcal{L}_t\varphi){\rm d}t + \left \langle \pi_t\left(\varphi g\right)- \pi_t(\varphi)\pi_t\left(g\right), C^{-1}\left({\rm d}Y_t - \pi_t(g){\rm d}t\right)\right \rangle
\end{equation*}
with $\mathcal{L}_t$ the infinitesimal generator of $X_t$, which in general might not be solvable analytically. In the linear case $f(x) = Ax, g(x) = Gx$, however, it is easy to see that if $\pi_0$ is Gaussian, also $\pi_t$ is a Gaussian distribution for all $t >0$ characterized by its mean
\[ m_t(l) := \int x_l \pi_t({\rm d}x)\]
and covariance matrix
\[ \Sigma_t(l_1, l_2) := \int  \left(x_{l_1} - m_t(l_1)\right)\left(x_{l_2} - m_t(l_2)\right)\bar{\pi}_t({\rm d}x)\]
whose evolution equations are given by the Kalman-Bucy filtering equations (cf. \cite{kalman1961})
\begin{align*}
{\rm d}m_t &= Am_t{\rm d}t + \Sigma_tG^TC^{-1}\left( {\rm d}Y_t - Gm_t{\rm d}t\right),\\
\frac{{\rm d}}{{\rm d}t} \Sigma_t &= A\Sigma_t + \Sigma_t A^T + Q - \Sigma_t G^TC^{-1}G\Sigma_t.
\end{align*}
Continuous-time ensemble-based Kalman-type filters form a generalization of the Kalman-Bucy filtering equations to the setting of nonlinear $f$ and $g$ in that they replace $m$ and $\Sigma$ by the first and second empirical moment of an ensemble. Observe that consequently these filters do not identify the full distribution $\pi_t$ and hence do not solve the optimal filtering problem but rather form an approximation scheme of $\pi_t$ up to second order moment. Popular examples are the Ensemble Kalman-Bucy Filter (EnKBF, e.g. \cite{bergemann2012}, \cite{kelly2014}) and the Ensemble Transform Kalman-Bucy Filter (ETKBF, e.g. \cite{amezcua2014}, \cite{bergemann2012}, \cite{deWiljes2018}) and recent advances in their mathematical analysis can be found for instance in \cite{deWiljes2018} in the full observations case, as well as in \cite{delMoral2018} for the linear setting as further summarized in \cite{abishop2020}.\\
\noindent
In our previous works \cite{lange2019} and \cite{lange2019b} we were able to rigorously derive both the EnKBF as well as the ETKBF as the limiting stochastic differential equations (SDE) of the discrete-time Ensemble Kalman Filter (EnKF, e.g. \cite{burgers1998}, \cite{evensen1994}), and the discrete-time Ensemble Square Root Filters (ESRF, e.g. \cite{tippett2003} and references therein) under rather restrictive assumptions on the coefficients: \cite{lange2019} covered the case of bounded $f$ and $g$, whereas in \cite{lange2019b} we considered the ESRF with linear observations and suitable deterministic perturbations chosen such that the resulting ensemble covariance matrices resembled the corresponding Kalman filtering equations. In a more general setting using linear observations, a formal derivation of the EnKBF and the ETKBF has been illustrated in \cite{bergemann2012} in finite dimensions, and \cite{kelly2014} set in Hilbert spaces. In this paper, we shall consider the case of nonlinear Lipschitz-continuous, unbounded $f$ and $g$ and provide a rigorous derivation of the EnKBF from the EnKF as well as the ETKBF from the ESRF algorithms.\\
Our analysis further necessitates certain bounds on the respective ensembles which resemble well-posedness and accuracy results. There exists theoretical and numerical evidence that the EnKF may experience blow-up known under the name of "catastrophic filter divergence" (see e.g. \cite{gottwald2013}, \cite{kelly2015}). Hence well-posedness and accuracy play an essential role in the analysis of the ensemble-based filters' performance and have therefore been of great interest in the literature: in the case of linear observations, consider for instance \cite{tong2015} and \cite{tong2016} providing uniform-in-time mean-squared estimates for both EnKF and ESRF using Lyapunov arguments, whereas in continuous time for stable signals imposing an observability criterion, the authors of \cite{delMoral2018} derive uniform-in-time bounds on higher-order moments of the EnKBF using Lyapunov techniques combined with methods from martingale and spectral theory, as well as concentration inequalities in \cite{delMoral2017} (also see the recent review \cite{abishop2020}). In the fully observed case with diagonal measurement covariance, the authors of \cite{kelly2014} show well-posedness of the EnKF as well as the EnKBF under certain assumptions on the dynamics of the underlying signal, and derive accuracy when employing variance inflation. In a similar setting with general Lipschitz-continuous $f$, well-posedness and accuracy of a variant of the ETKBF was shown in \cite{deWiljes2018}. Though we employ similar techniques as used in the articles presented above, the bounds needed for our analysis are still of different and more general type since they hold without further restrictions on the underlying filtering setting nor under further structural or algorithmic assumptions.\\
We present our arguments as follows: starting off with introducing the above continuous-time filtering algorithms in Section \ref{ContAlgo}, we show in Section \ref{ContWellPosed} boundedness of the continuous-time filters in ensemble-mean-squared sense yielding well-posedness results, as well as boundedness of the approximation error of the filters to the truth as formalized in Theorem \ref{AccCont}.\\
In Section \ref{DiscrAlgo}, we introduce the above discrete-time filtering algorithms which allow for a similar well-posedness and accuracy result as given in Theorem \ref{WellposedEnKFs} and in Theorem \ref{AccDiscr}, respectively. The final section is then concerned with the main part of this paper which is the rigorous proof of the above continuous time limit claim made precise in Theorem \ref{mainResult1} and Theorem \ref{mainResult2}.

\subsection{Notation}
We use the standard notation that for a vector $x \in \mathds{R}^{n}$ and a matrix $A \in \mathds{R}^{n \times m}$ both $x^T$ and $A^T$ denote the respective transpose, and $\|x\|$ a vector norm on $\mathds{R}^{n}$ as well as $\|A\|$ and $\|A\|_F$ the spectral norm and the Frobenius norm of $A$, respectively. We further introduce the notation $\lesssim$ when the estimation $\leq$ holds with a multiplicative constant, e.g. arising in the context of applying the Cauchy-Schwarz inequality. The Lipschitz constant of a Lipschitz continuous function $f$ is denoted by $\|f\|_{\text{Lip}}$. Furthermore denote
\[ |f|_{+} := \sup_{x \neq y} \frac{\langle x - y, f(x) - f(y) \rangle}{\|x - y\|^2}\]
where clearly $|f|_{+} \leq \|f\|_{\text{Lip}}$. One can show that further there exists a constant $C_f$ such that
\[ \langle x, f(x) \rangle \leq C_f\left(1+\|x\|^2\right)\]
and a constant $\tilde{C}_f$ such that
\[ \|f(x)\|^2\leq \tilde{C}_f\left(1+\|x\|^2\right).\]

\section{Continuous-time ensemble-based Kalman-type filtering algorithms}\label{ContAlgo}
\noindent
As described in the introduction, continuous-time ensemble-based Kalman-type filtering algorithms propose continuous-time evolution equations of an ensemble $X_t^{(i)}, 1 \leq i \leq M$, such that the ensemble mean
\begin{equation*}
\bar{x}_t := \frac{1}{M} \sum_{i=1}^M X_t^{(i)}
\end{equation*}
and the empirical covariance matrix
\begin{equation*}
P_t := \frac{1}{M-1}\sum_{i=1}^M \left(X_t^{(i)} - \bar{x}_t\right)\left(X_t^{(i)}-\bar{x}_t\right)^T
\end{equation*}
approximate the Kalman-Bucy filtering equations for mean and covariance matrix of the conditional distribution $\pi_t$ corresponding to the setting \eqref{X}-\eqref{Y}. The EnKBF specifies the ensemble via
\begin{equation}\label{EnKBF}
{\rm d}X_t^{(i)} = f\left(X_t^{(i)}\right){\rm d}t + Q^{\frac{1}{2}}{\rm d}W_t^{(i)} + K_t\left({\rm d}Y_t + C^{\frac{1}{2}}{\rm d}V_t^{(i)} - g\left(X_t^{(i)}\right){\rm d}t\right)
\end{equation}
where
\begin{equation}
K_t := \frac{1}{M-1}E_t\mathcal{G}_t^TC^{-1}
\end{equation}
is the so-called \emph{Kalman gain} with
\begin{align*}
E_t &:= \left[X_t^{(i)}-\bar{x}_t\right]_{i=1}^M,\\
\mathcal{G}_t &:= \left[g\left(X_t^{(i)}\right)-\bar{g}_t\right]_{i=1}^M, \quad \bar{g}_t := \frac{1}{M}\sum_{i=1}^M g\left(X_t^{(i)}\right),
\end{align*}
and $W^{(i)}$ and $V^{(i)}$ are $M$ independent Brownian motions. Ensemble mean and covariance matrix thus satisfy (up to explosion time) the following evolution equations:
\begin{equation*}
{\rm d}\bar{x}_t = \bar{f}_t{\rm d}t + Q^{\frac{1}{2}}{\rm d}\bar{w}_t + K_t\left({\rm d}Y_t + C^{\frac{1}{2}}{\rm d}\bar{v}_t - \bar{g}_t{\rm d}t\right)
\end{equation*}
and
\begin{equation*}
\begin{aligned}
P_t &= \left[\frac{1}{M-1}\sum_{i=1}^M \left(f\left(X_t^{(i)}\right)-\bar{f}_t\right)\left(X_t^{(i)}-\bar{x}_t\right)^T + \left(X_t^{(i)}-\bar{x}_t\right)\left(f\left(X_t^{(i)}\right)-\bar{f}_t\right)\right]{\rm d}t\\
&\quad + \left(Q + K_tCK_t^T -\frac{2}{M-1}K_t\mathcal{G}_tE_t^T\right){\rm d}t + {\rm d}\mathcal{M}^{(1)}_t
\end{aligned}
\end{equation*}
with the martingale
\begin{equation*}
\begin{aligned}
{\rm d}\mathcal{M}^{(1)}_t &:= \frac{1}{M-1}\sum_{i=1}^M \left(Q^{\frac{1}{2}}{\rm d}\left(W_t^{(i)}-\bar{w}_t\right)+K_tC^{\frac{1}{2}}{\rm d}\left(V_t^{(i)}-\bar{v}_t\right)\right)\left(X_t^{(i)}-\bar{x}_t\right)^T\\
&\hspace{2.5cm} + \left(X_t^{(i)}-\bar{x}_t\right)\left(Q^{\frac{1}{2}}{\rm d}\left(W_t^{(i)}-\bar{w}_t\right)+K_tC^{\frac{1}{2}}{\rm d}\left(V_t^{(i)}-\bar{v}_t\right)\right)^T.
\end{aligned}
\end{equation*}
\noindent
The ETKBF, on the other hand, consists of an ensemble of the following form:
\begin{equation}\label{ETKBF}
{\rm d}X_t^{(i)} = f\left(X_t^{(i)}\right){\rm d}t + Q^{\frac{1}{2}}{\rm d}W_t^{(i)} + K_t\left({\rm d}Y_t - \frac{1}{2}\left(g\left(X_t^{(i)}\right)+\bar{g}_t\right){\rm d}t\right)
\end{equation}
with (up to explosion time) ensemble mean
\begin{equation*}
{\rm d}\bar{x}_t = \bar{f}_t{\rm d}t + Q^{\frac{1}{2}}{\rm d}\bar{w}_t + K_t\left({\rm d}Y_t - \bar{g}_t{\rm d}t\right)
\end{equation*}
and covariance matrix
\begin{equation*}
\begin{aligned}
{\rm d}P_t &= \left[\frac{1}{M-1}\sum_{i=1}^M \left(f\left(X_t^{(i)}\right)-\bar{f}_t\right)\left(X_t^{(i)}-\bar{x}_t\right)^T + \left(X_t^{(i)}-\bar{x}_t\right)\left(f\left(X_t^{(i)}\right)-\bar{f}_t\right)\right]{\rm d}t\\
&\quad + \left(Q-\frac{1}{M-1}K_t\mathcal{G}_tE_t^T\right){\rm d}t + {\rm d}\mathcal{M}^{(2)}_t
\end{aligned}
\end{equation*}
with the martingale
\begin{equation*}
\begin{aligned}
{\rm d}\mathcal{M}^{(2)}_t &:= \frac{1}{M-1}\sum_{i=1}^M \left(Q^{\frac{1}{2}}{\rm d}\left(W_t^{(i)}-\bar{w}_t\right)\right)\left(X_t^{(i)}-\bar{x}_t\right)^T \\
&\hspace{2.5cm}+ \left(X_t^{(i)}-\bar{x}_t\right)\left(Q^{\frac{1}{2}}{\rm d}\left(W_t^{(i)}-\bar{w}_t\right)\right)^T.
\end{aligned}
\end{equation*}
For both \eqref{EnKBF} and \eqref{ETKBF} we do not immediately obtain existence of global (strong) solutions since the coefficients are only locally Lipschitz and of cubic growth in the ensemble variable. Nevertheless we are able to derive certain bounds which resemble well-posedness and accuracy results of both filters. These will be summarized in the next sections. In this context we will use that the observation process admits the following representation
\begin{equation}\label{YXref}
{\rm d}Y_t = g\left(X_t^{\text{ref}}\right){\rm d}t + C^{\frac{1}{2}}{\rm d}V_t
\end{equation}
where $X^{\text{ref}}$ denotes the reference trajectory of the signal that generates the observations and we assume that
\begin{equation*}
\sup_{t \in [0,T]} \mathds{E}\left[\left\|X_t^{\text{ref}}\right\|^2\right] < \infty.
\end{equation*}

\subsection{Well-posedness}\label{ContWellPosed}
As already discussed in \cite{langeStannat2020} in case of the ETKBF, for both continuous-time filtering algorithms we decompose the ensemble member into centered ensemble member
\begin{equation*}
\hat{X}_t^{(i)} := X_t^{(i)}-\bar{x}_t
\end{equation*}
and ensemble mean, and consider the trace of $P_t$ defined by
\begin{equation*}
\text{tr}(P_t) := \sum_{k=1}^d P_t(k,k) = \frac{1}{M-1}\sum_{i=1}^M \left\|\hat{X}_t^{(i)}\right\|^2
\end{equation*}
which up to explosion time $\xi$ of the filter satisfies the following inequality
\begin{equation}\label{ContTrace}
{\rm d}\text{tr}(P_t) \leq \left(2|f|_{+}\text{tr}(P_t)+\text{tr}(Q)\right){\rm d}t + {\rm d}\mathcal{N}_t
\end{equation}
where
\begin{equation*}
\begin{aligned}
{\rm d}\mathcal{N}_t &:= {\rm d}\text{tr}\left(\mathcal{M}_t^{(1)}\right) = \frac{2}{M-1}\sum_{i=1}^M \left \langle Q^{\frac{1}{2}}{\rm d}\left(W_t^{(i)}-\bar{w}_t\right)\right.\\
&\hspace{5cm}\left.+K_tC^{\frac{1}{2}}{\rm d}\left(V_t^{(i)}-\bar{v}_t\right),X_t^{(i)}-\bar{x}_t\right\rangle
\end{aligned}
\end{equation*}
in case of the EnKBF, and
\begin{equation*}
{\rm d}\mathcal{N}_t := {\rm d}\text{tr}\left(\mathcal{M}_t^{(2)}\right) = \frac{2}{M-1}\sum_{i=1}^M \left \langle Q^{\frac{1}{2}}{\rm d}\left(W_t^{(i)}-\bar{w}_t\right),X_t^{(i)}-\bar{x}_t\right \rangle
\end{equation*}
in case of the ETKBF. Thus by the stochastic Gronwall lemma (cf. \cite{scheutzow2013}) it holds
\begin{equation}\label{WellPosPCont}
\mathds{E}\left[\sup_{t \in [0,T \wedge \xi]}\sqrt{\text{tr}(P_t)}\right] \lesssim e^{\|f\|_{\text{Lip}}T}\text{tr}(Q).
\end{equation}
Similarly we deduce
\begin{equation}\label{WellPosMCont}
\mathds{E}\left[\sup_{t \in [0,T \wedge \xi]} \sqrt{\|\bar{x}_t\|}\right] < \infty
\end{equation}
which in total implies $\mathds{P}\left[\xi \leq T\right] = 0$ for all $T$ yielding well-posedness of the filters as well as existence of a global solution to \eqref{EnKBF}, resp. \eqref{ETKBF}.\\
A similar result which we will make use of is the following:
\begin{theorem}\label{WellposedEnKBF}
There exists a constant $\mathcal{C} = \mathcal{C}(T,n;f,g,C,Q,M)$ such that the EnKBF as well as the ETKBF satisfy
\begin{equation}\label{WellPosMemCont}
\mathds{E}\left[\sup_{t \in [0, T \wedge \tau_n]} \sum_{i=1}^M \left\|X_t^{(i)}\right\|^2\right] \leq \mathcal{C}
\end{equation}
where the stopping time
\begin{equation*}
\tau_n := \inf\left\{ t > 0: {\rm tr}(P_t) > n\right\}
\end{equation*}
satisfies
\begin{equation}
\limsup_{n \rightarrow \infty} \mathds{P}\left[\tau_n \leq T\right] = 0
\end{equation}
for every $T >0$.
\end{theorem}
\noindent
For the proof see Appendix \ref{appWellposedEnKBF}. Also compare these results with the uniform-in-time moment estimates for the EnKBF in Section 7.2 in \cite{delMoral2017} following a similar strategy of proof. In view of our proceeding analysis, however, we need the specific estimates \eqref{WellPosPCont}-\eqref{WellPosMemCont}.

\subsection{Accuracy-type bounds}
Let $X^{\text{ref}}$ denote a reference trajectory given by
\begin{equation*}
{\rm d}X_t^{\text{ref}} = f\left(X_t^{\text{ref}}\right){\rm d}t + Q^{\frac{1}{2}}{\rm d}W_t.
\end{equation*}
Then we obtain the following result:
\begin{theorem}\label{AccCont}
Let $\left(X_t^{(i)}\right)_{t \geq 0}$, $1\leq i \leq M$, solve \eqref{EnKBF} or \eqref{ETKBF}. Then there exist constants $\mathcal{C}^{(1)} = \mathcal{C}^{(1)}(T,n;f,g,C,Q,M)$ and \\$\mathcal{C}^{(2)} = \mathcal{C}^{(2)}(T,n; f,g,C,Q,M)$ such that
\begin{equation}
\mathds{E}\left[\sup_{t \in [0, T\wedge \tau_n]} \sum_{i=1}^M \left\|X_t^{{\rm ref}} - X_t^{(i)}\right\|^2\right] \leq \mathcal{C}^{(1)}\mathds{E}\left[\sum_{i=1}^M\left\|X_0^{{\rm ref}}-X_0^{(i)}\right\|^2\right] + \mathcal{C}^{(2)}.
\end{equation}
\end{theorem}
\noindent
For the proof see Appendix \ref{appAccCont}.

\section{Discrete-time ensemble-based Kalman-type filtering algorithms}\label{DiscrAlgo}
\noindent
Solving the filtering problem from the last section in the applications necessitates a suitable discretization of the underlying setting \eqref{X}-\eqref{Y}. In our context we choose the corresponding Euler-Maruyama time-discretizations
\begin{align*}
X_{t_k}^{h} &= X_{t_{k-1}}^{h} + hf\left(X_{t_{k-1}}^{h}\right) + Q^{\frac{1}{2}}W_k^{h},\\
Y_{t_k}^{h} &= Y_{t_{k-1}}^{h} + hg\left(X_{t_{k-1}}^{h}\right) + C^{\frac{1}{2}}V_k^{h}
\end{align*}
on the equidistant partition $0 = t_0 < t_1 < ... < t_L = T$, $t_k = t_{k-1} + h$, of $[0,T]$ for $T>0$ fixed where $W_k^{h} \overset{d}{=} W_{t_k}-W_{t_{k-1}} \sim \mathcal{N}(0,h{\rm Id})$ and $V_k^{h} \overset{d}{=} V_{t_k}-V_{t_{k-1}} \sim \mathcal{N}(0,h{\rm Id})$. The corresponding discrete-time filtering problem is to identify the conditional distribution 
\[\pi_k(x) := \mathds{P}\left[X_{t_k}^{h} = x | \mathcal{Y}_{0:k}\right]\]
where in our case $\mathcal{Y}_{0:k} := \sigma\left(\Delta Y_1, ..., \Delta Y_k\right), \Delta Y_j := Y_{t_j} - Y_{t_{j-1}}, \Delta Y_0 = 0$, whose computation in the linear, Gaussian case reduces to the well-known Kalman filtering equations (cf. \cite{kalman1960}) for its mean and covariance matrix. Note that in $\Delta Y_j$  we use the observations coming from the measurement device as opposed to the measurements prescribed by the discretized model since the former are the actual data that are used in the applications.\\
In the proceeding analysis we will make use of the following notation: for $t \in [t_{k-1}, t_k)$ denote
\[\eta(t) := t_{k-1}, \eta_{+}(t):= t_k, \nu(t) := k-1, \nu_{+}(t) := k.\]
\subsection{Algorithms}\label{algos}
Discrete-time ensemble-based Kalman-type filtering algorithms iterate an ensemble of $M$ particles $X_k^{(i),a}, 1 \leq i \leq M,$ in such a way that its first and second empirical moments
\begin{align*}
\bar{x}_k^{a} &:= \frac{1}{M}\sum_{i=1}^M X_k^{(i),a},\\
P_k^{a} &:= \frac{1}{M-1}\sum_{i=1}^M \left(X_k^{(i),a} - \bar{x}_k^{a}\right)\left(X_k^{(i),a} - \bar{x}_k^{a}\right)^T
\end{align*}
form approximations of the Kalman filtering equations for first and second moment of the conditional distribution $\pi_k$. They iterate the ensemble in two steps: in a \emph{forecast step}, the ensemble of a previous estimation cycle $X_{k-1}^{(i),a}, 1\leq i \leq M$, is propagated forward according to the model equations forming the \emph{forecast ensemble} given by
\begin{equation*}
X_k^{(i),f} = X_{k-1}^{(i),a} + hf\left(X_{k-1}^{(i),a}\right) + Q^{\frac{1}{2}}W_k^{(i)}
\end{equation*}
where $W_k^{(i)} \sim \mathcal{N}(0, h{\rm Id})$ form $M$ independent samples of $W_k^{h}$. On the arrival of an observation $\Delta Y_k$ at time $t_k$, each of these forecasts will be updated to yield an \emph{analyzed ensemble} $\left(X_k^{(i),a}\right)$: let
\begin{equation*}
K_k := \frac{1}{M-1}E_k^{f}\left(\mathcal{G}_k^{f}\right)^T\left(C+ \frac{h}{M-1}\mathcal{G}_k^{f}\left(\mathcal{G}_k^{f}\right)^T\right)^{-1}
\end{equation*}
denote the so-called \emph{Kalman gain} where
\begin{align*}
E_k^{f} &:= \left[X_k^{(i),f} - \bar{x}_k^{f}\right]_{i=1}^M \in \mathds{R}^{d \times M},\\
\mathcal{G}_k^{f} &:= \left[g\left(X_k^{(i),f}\right) - \bar{g}_k^{f}\right]_{i=1}^M \in \mathds{R}^{p \times M}, \quad \bar{g}_k^{f} := \frac{1}{M}\sum_{i=1}^M g\left(X_k^{(i),f}\right),
\end{align*}
then this paper is concerned with the following filtering algorithms:
\begin{itemize}
\item the Ensemble Kalman Filter (short: EnKF)
\begin{equation*}
X_k^{(i),a} = X_k^{(i),f} + K_k\left(\Delta Y_k + C^{\frac{1}{2}}V_k^{(i)} - hg\left(X_k^{(i),f}\right)\right)
\end{equation*}
where $V_k^{(i)} \sim \mathcal{N}(0, h{\rm Id})$ form $M$ independent samples of $V_k^{h}$
\item Ensemble Square Root Filters (short: ESRF)
\begin{equation*}
X_k^{(i),a} = \bar{x}_k^{a} + E_k^{a}e_i
\end{equation*}
where
\begin{equation*}
\bar{x}_k^{a} = \bar{x}_k^{f} + K_k\left(\Delta Y_k - h\bar{g}_k^{f}\right)
\end{equation*}
and $E_k^{a}$ is a transformation of $E_k^{f}$ such that
\begin{equation*}
P_k^{a} := \frac{1}{M-1}E_k^{a}\left(E_k^{a}\right)^T \overset{!}{=} P_k^{f} - \frac{1}{M-1}K_k\mathcal{G}_k^{f}\left(E_k^{f}\right)^T
\end{equation*}
and $e_i$ denotes the $i$-th standard normal basis vector in $\mathds{R}^M$.
\end{itemize}
For the class of ESRF, we focus on the three most popular algorithms summarized in \cite{tippett2003} in the case of linear observations: the Ensemble Adjustment Kalman Filter (EAKF, cf. \cite{anderson2001})
\begin{equation*}
E_k^{a} = A_kE_k^{f},
\end{equation*}
the Ensemble Transform Kalman Filter (ETKF, cf. \cite{bishop2001})
\begin{equation*}
E_k^{a} = E_k^{f}T_k,
\end{equation*}
with transformations $A_k$ and $T_k$ specified below, and the unperturbed filter by \cite{whitaker2002}
\begin{equation}\label{WH}
E_k^{a} = \left({\rm Id} - \tilde{K}_kG\right)E_k^{f}
\end{equation}
where
\begin{equation*}
\tilde{K}_k := P_k^{f}G^T\left(C+hGP_k^{f}G^T\right)^{-\frac{1}{2}}\left(C^{\frac{1}{2}}+\left(C+hGP_k^{f}G^T\right)^{\frac{1}{2}}\right)^{-1}.
\end{equation*}
As we have shown in our recent paper \cite{langeStannat2020}, the transformations $A_k$ and $T_k$ allow for the following analytic representations: using the integral representation
\begin{equation*}
\sqrt{P^{-1}} = \frac{1}{\sqrt{\pi}}\int_0^{\infty} \frac{1}{\sqrt{t}}e^{-tP}{\rm d}t
\end{equation*}
for symmetric positive semidefinite matrix $P$, we obtain
\begin{align*}
A_k &= \sqrt{P_k^{f}}\left({\rm Id} + h\sqrt{P_k^{f}}G^TC^{-1}G\sqrt{P_k^{f}}\right)^{-\frac{1}{2}}\sqrt{P_k^{f}}^{-1},\\
T_k &= \left({\rm Id} + \frac{h}{M-1} \left(E_k^{f}\right)^TG^TC^{-1}GE_k^{f}\right)^{-\frac{1}{2}}
\end{align*}
where $\sqrt{P_k^{f}}$ denotes the symmetric positive semidefinite square root of $P_k^{f}$ and $\sqrt{P_k^{f}}^{-1}$ its pseudo inverse. Observe that these transformations are adjoint in the sense that
\begin{equation*}
A_kE_k^{f} = \frac{1}{\sqrt{\pi}}\int_0^{\infty} \frac{e^{-t}}{\sqrt{t}}e^{-thP_k^{f}G^TC^{-1}G}{\rm d}tE_k^{f} = E_k^{f}T_k
\end{equation*}
thus analytically we will not distinguish between the two filters.\\
In the setting of nonlinear observations, the above three ESRF transformations generalize to the following transformations
\begin{equation*}
E_k^{a} = E_k^{f} - \frac{h}{2(M-1)}E_k^{f}\left(\mathcal{G}_k^{f}\right)^TC^{-1}\mathcal{G}_k^{f} + R_k^{h}
\end{equation*}
for EAKF/ETKF where
\begin{equation}\label{Rkh}
R_k^{h} := \frac{1}{\sqrt{\pi}}\int_0^{\infty} \frac{e^{-t}}{\sqrt{t}}E_k^{f}\left(e^{-\frac{th}{M-1}\left(\mathcal{G}_k^{f}\right)^TC^{-1}\mathcal{G}_k^{f}} - {\rm Id} + \frac{th}{M-1}\left(\mathcal{G}_k^{f}\right)^TC^{-1}\mathcal{G}_k^{f}\right){\rm d}t,
\end{equation}
as well as to the transformation \eqref{WH} in case of the unperturbed filter where this time
\begin{equation*}
\begin{aligned}
\tilde{K}_k &:= \frac{1}{M-1}E_k^{f}\left(\mathcal{G}_k^{f}\right)^T\\
&\quad\times\left(C+\frac{h}{M-1}\mathcal{G}_k^{f}\left(\mathcal{G}_k^{f}\right)^T\right)^{-\frac{1}{2}}\left(C^{\frac{1}{2}}+\left(C+\frac{h}{M-1}\mathcal{G}_k^{f}\left(\mathcal{G}_k^{f}\right)^T\right)^{\frac{1}{2}}\right)^{-1}.
\end{aligned}
\end{equation*}
In summary: the filtering algorithms considered in this paper are the EnKF
\begin{align}
X_k^{(i),f} &= X_{k-1}^{(i),a} +hf\left(X_{k-1}^{(i),a}\right) + Q^{\frac{1}{2}}W_k^{(i)},\label{EnKFF}\\
X_k^{(i),a} &= X_k^{(i),f} K_k\left(\Delta Y_k + C^{\frac{1}{2}}V_k^{(i)} - hg\left(X_k^{(i),f}\right)\right)\label{EnKFA}
\end{align}
as well as the above ESRF algorithms
\begin{align}
X_k^{(i),f} &= X_{k-1}^{(i),a} +hf\left(X_{k-1}^{(i),a}\right) + Q^{\frac{1}{2}}W_k^{(i)},\label{ESRFF}\\
X_k^{(i),a} &= X_k^{(i),f} -h\hat{K}_kg\left(X_k^{(i),f}\right) - h\left(K_k - \hat{K}_k\right)\bar{g}_k^{f} + K_k\Delta  Y_k + \mathcal{R}_k^{(i)}\label{ESRFA}
\end{align}
where
\begin{itemize}
\item for EAKF/ETKF
\begin{align*}
\hat{K}_k &:= \frac{1}{2(M-1)}E_k^{f}\left(\mathcal{G}_k^{f}\right)^TC^{-1},\\
\mathcal{R}_k^{(i)} &:= R_k^{h}e_i
\end{align*}
\item for the unperturbed filter
\begin{align*}
\hat{K}_k &:= \tilde{K}_k,\\
\mathcal{R}_k^{(i)} &= 0.
\end{align*}
\end{itemize}

\subsection{Well-posedness}
\begin{theorem}\label{WellposedEnKFs}
The stopping time
\begin{equation*}
\tau_n^{h} := \inf\left\{t >0: {\rm tr}\left(P_{\nu(t)}^{f}\right)>n\right\},
\end{equation*}
satisfies
\begin{equation}\label{ConvBarTau}
\limsup_{n\rightarrow \infty}\limsup_{h \rightarrow 0}\mathds{P}\left[\tau_n^{h} \leq T\right] = 0
\end{equation}
for every $T >0$.
Furthermore there exists a constant \\$\mathcal{D}=\mathcal{D}(T,n;f,g,C,Q,M)$ such that the discrete-time filtering algorithms presented in Section \ref{algos} satisfy
\begin{equation}\label{WellPosMemDisc}
\mathds{E}\left[\int_0^{T\wedge \tau_n^{h}} \sum_{i=1}^M \left\|X_{\nu(s)}^{(i),a}\right\|^2{\rm d}s\right] \leq \mathcal{D}.
\end{equation}
\end{theorem}
\noindent
For the proof see Appendix \ref{appWellposedEnKFs}. A similar result can be found in \cite{tong2016} specifying a uniform-in-time bound on the second moments of the ensemble in the case of linear, full-rank observations. Our proceeding analysis, on the other hand, uses the more specific estimate \eqref{WellPosMemDisc}.

\subsection{Accuracy-type bounds}
As in the continuous-time case, consider a reference trajectory $X^{\text{ref}}$ given by
\begin{equation*}
{\rm d}X_t^{\text{ref}} = f\left(X_t^{\text{ref}}\right){\rm d}t + Q^{\frac{1}{2}}{\rm d}W_t,
\end{equation*}
then assuming
\begin{equation*}
\sup_{t \in [0,T]} \mathds{E}\left[\left\|X_t^{\text{ref}}\right\|^2\right] < \infty
\end{equation*}
we obtain the following result:
\begin{theorem}\label{AccDiscr}
Let $\left(X_k^{(i),a}\right)_{k \geq 0}$, $1\leq i \leq M$, denote the EnKF \eqref{EnKFF}-\eqref{EnKFA} or the ESRF algorithms \eqref{ESRFF}-\eqref{ESRFA}. Then there exist constants $\mathcal{D}^{(1)} = \mathcal{D}^{(1)}(T,n;f,g,C,Q,M)$ and $\mathcal{D}^{(2)} = \mathcal{D}^{(2)}(T,n; f,g,C,Q,M)$ such that
\begin{equation}
\mathds{E}\left[\sup_{t \in [0, T\wedge \tau_n^{h}]}\sum_{i=1}^M\left\|X_t^{{\rm ref}}-X_{\nu(t)}^{(i),a}\right\|^2\right] \leq \mathcal{D}^{(1)}\mathds{E}\left[\sum_{i=1}^M\left\|X_0^{{\rm ref}}-X_0^{(i),a}\right\|^2\right]+\mathcal{D}^{(2)}.
\end{equation}
\end{theorem}
\noindent
For the proof see Appendix \ref{appAccDiscr}.

\section{Continuous time limit}
\noindent
In this section we aim to show that the EnKBF and the ETKBF can be derived from the EnKF and the ESRF, respectively, in that we investigate convergence of the ensemble-mean squared distance
\[\sum_{i=1}^M \left\|X_{\nu(t)}^{(i),a} - X_t^{(i)}\right\|^2, \quad t \in [0,T],\]
as $h$ converges to 0. Using \eqref{YXref}, these differences allow for the following decompositions: in case of the EnKF compared with the EnKBF
\begin{equation}\label{decompDiffEnKF}
\begin{aligned}
&X_{\nu(t)}^{(i),a} - X_t^{(i)} = X_{\nu(t)}^{(i),a} - X_{\eta(t)}^{(i)} +X_{\eta(t)}^{(i)}-X_t^{(i)}\\
&= X_0^{(i),a} - X_0^{(i)}\\
&\quad +\int_0^{\eta(t)}f\left(X_{\nu(s)}^{(i),a}\right) - f\left(X_s^{(i)}\right) \\
&\hspace{2cm}+ \left(K_{\nu_{+}(s)}-K_s\right)\left(g\left(X_s^{\text{ref}}\right) - g\left(X_s^{(i)}\right)\right)\\
&\hspace{2cm} -K_{\nu_{+}(s)}\left(g\left(X_{\nu_{+}(s)}^{(i),f}\right)-g\left(X_s^{(i)}\right)\right){\rm d}s\\
&\quad + \int_0^{\eta(t)}\left(K_{\nu_{+}(s)}-K_s\right)C^{\frac{1}{2}}{\rm d}V_s\\
&\quad + \int_0^{\eta(t)}\left(K_{\nu_{+}(s)}-K_s\right)C^{\frac{1}{2}}{\rm d}V_s^{(i)}\\
&\quad + X_{\eta(t)}^{(i)} - X_t^{(i)}.
\end{aligned}
\end{equation}
as well as in the case of the ESRF compared with the ETKBF
\begin{equation}\label{decompDiffESRF}
\begin{aligned}
&X_{\nu(t)}^{(i),a} - X_t^{(i)}\\
&=X_0^{(i),a} - X_0^{(i)}\\
&\quad +\int_0^{\eta(t)}f\left(X_{\nu(s)}^{(i),a}\right) - f\left(X_s^{(i)}\right) \\
&\hspace{2cm} + \left(\hat{K}_{\nu_{+}(s)}-\frac{1}{2}K_s\right)\left(g\left(X_s^{\text{ref}}\right) - g\left(X_s^{(i)}\right)\right)\\
&\hspace{2cm}+\left(K_{\nu_{+}(s)}-\hat{K}_{\nu_{+}(s)} - \frac{1}{2}K_s\right)\left(g\left(X_s^{\text{ref}}\right)-\bar{g}_s\right)\\
&\hspace{2cm} -\hat{K}_{\nu_{+}(s)}\left(g\left(X_{\nu_{+}(s)}^{(i),f}\right)-g\left(X_s^{(i)}\right)\right) \\
&\hspace{2cm}- \left(K_{\nu_{+}(s)} - \hat{K}_{\nu_{+}(s)}\right)\left(\bar{g}_{\nu_{+}(s)}^{f} - \bar{g}_s\right){\rm d}s\\
&\quad + \int_0^{\eta(t)}\left(K_{\nu_{+}(s)}-K_s\right)C^{\frac{1}{2}}{\rm d}V_s\\
&\quad + \sum_{k=1}^{\nu(t)} \mathcal{R}_k^{(i)} + X_{\eta(t)}^{(i)} - X_t^{(i)}.
\end{aligned}
\end{equation}
Our proof makes use of the following a priori estimates:
\begin{lemma}\label{estimatesKGain}
The operator norms of the discrete-time and continuous-time Kalman gains respectively satisfy
\begin{align*}
\left\|K_{\nu(t)}\right\| &\leq \kappa_K{\rm tr}\left(P_{\nu(t)}^{f}\right)\\
\left\|\hat{K}_{\nu(t)}\right\| &\leq \frac{1}{2}\kappa_K{\rm tr}\left(P_{\nu(t)}^{f}\right)\\
\left\|K_t\right\| &\leq \kappa_K{\rm tr}\left(P_t\right)
\end{align*}
where $\kappa_K := \|g\|_{\text{Lip}}\left\|C^{-1}\right\|$. Further let $\Delta_t$ denote 
\[ \left\|K_{\nu(t)} - K_t\right\|, \left\|\hat{K}_{\nu(t)} - \frac{1}{2}K_t\right\|, \text{ or } \left\|K_{\nu(t)}-\hat{K}_{\nu(t)}-\frac{1}{2}K_t\right\|,\]
then $\Delta_t$ satisfies the following estimate
\begin{equation}\label{diffGains}
\Delta_t \leq \kappa_{\text{diff}}\left( h{\rm tr}\left(P_{\nu(t)}^{f}\right)^2 + \left({\rm tr}\left(P_{\nu(t)}^{f}\right)^{\frac{1}{2}}+{\rm tr}(P_t)^{\frac{1}{2}}\right)\left(\sum_{i=1}^M \left\|X_{\nu(t)}^{(i),f} - X_t^{(i)}\right\|^2\right)^{\frac{1}{2}}\right)
\end{equation}
for a constant $\kappa_{\text{diff}}$ different for each case of $\Delta_t$, but independent of $h$.
\end{lemma}
\noindent
For the proof see Appendix \ref{appLemma}. The convergence of the EnKF to the EnKBF is thence made precise in the following:
\newpage
\begin{theorem}\label{mainResult1}
Let $\left(W_t^{(i)}\right), \left(V_t^{(i)}\right), i=1,...,M,$ be independent Brownian motions and denote $W_k^{(i)} := W_{t_k}^{(i)} -W_{t_{k-1}}^{(i)}$ and $V_k^{(i)} := V_{t_k}^{(i)} - V_{t_{k-1}}^{(i)}$. Let $\left(X_k^{(i),f}, X_k^{(i),a}\right), i=1,...,M$, denote the EnKF \eqref{EnKFF}-\eqref{EnKFA} with respect to $\left(W_k^{(i)}\right), \left(V_k^{(i)}\right)$ and $\left(X_t^{(i)}\right), i=1,...,M$, a strong solution of the EnKBF-SDE \eqref{EnKBF} with respect to $W_t^{(i)}$ and $V_t^{(i)}$. Then $X_t^{(i)}$ forms a continuous time limit of $X_k^{(i),f}, X_k^{(i),a}$ in the sense that if
\begin{equation}
\mathbb{E}\left[\sum_{i=1}^M \left\|X_0^{(i),a}-X_0^{(i)}\right\|^2\right]\lesssim h,
\end{equation}
then there exists a constant $\mathcal{C}= \mathcal{C}(T,M,n,m_n)$ such that
\begin{equation}
\mathds{E}\left[\sup_{t \in [0, T \wedge \tau_n \wedge \tau_n^{h} \wedge \tau_n^{{\rm ref}}]}\sum_{i=1}^M\left\|X_{\nu(t)}^{(i),a} - X_t^{(i)}\right\|^2\right] \leq \mathcal{C}h
\end{equation}
where for the stopping time
\begin{equation*}
\tau_n^{{\rm ref}} := \inf\left\{ t > 0: \sum_{i=1}^M \left\|X_t^{{\rm ref}}-X_t^{(i)}\right\|^2 > m_n\right\}
\end{equation*}
the value $m_n$ can be chosen in such a way that
\begin{equation}
\limsup_{n \rightarrow \infty} \mathds{P}\left[\tau_n^{{\rm ref}} \leq T \wedge \tau_n\right] =0.
\end{equation}
for every $T>0$.
\end{theorem}
\begin{proof}
Recall from Theorem \ref{AccCont} that there exist constants $\mathcal{C}^{(1)}$ and $\mathcal{C}^{(2)}$ such that
\begin{equation*}
\mathds{E}\left[\sup_{t \in [0, T\wedge \tau_n]} \sum_{i=1}^M \left\|X_t^{\text{ref}} - X_t^{(i)}\right\|^2\right] \leq \mathcal{C}^{(1)}\mathds{E}\left[\sum_{i=1}^M\left\|X_0^{\text{ref}}-X_0^{(i)}\right\|^2\right] + \mathcal{C}^{(2)}.
\end{equation*}
By their explicit form given in \eqref{RefContDetail} in Appendix \ref{appAccCont}, we may choose $m_n$ such that
\[ \max\left(\mathcal{C}^{(1)},\mathcal{C}^{(2)}\right) \in o(m_n)\]
which then yields by Markov's inequality
\begin{align*}
\mathds{P}\left[\tau_n^{\text{ref}} \leq T \wedge \tau_n\right] &= \mathds{P}\left[\sum_{i=1}^M \left\|X_{T\wedge \tau_n}^{\text{ref}} - X_{T\wedge \tau_n}^{(i)}\right\|^2 > m_n\right]\\
&\leq \frac{1}{m_n} \mathds{E}\left[\sup_{t \in [0,T \wedge \tau_n]} \sum_{i=1}^M \left\|X_t^{\text{ref}} - X_t^{(i)}\right\|^2\right] \longrightarrow 0, n \rightarrow \infty.
\end{align*}
Next we consider \eqref{decompDiffEnKF} in the ensemble-mean squared norm and apply the Cauchy-Schwarz inequality to obtain
\begin{equation}
\begin{aligned}
&\sum_{i=1}^M \left\|X_{\nu(t)}^{(i),a} - X_t^{(i)}\right\|^2\\
&\lesssim \sum_{i=1}^M \left\|X_0^{(i),a} - X_0^{(i)}\right\|^2\\
&\quad + \eta(t) \int_0^{\eta(t)} \|f\|_{\text{Lip}}^2\sum_{i=1}^M\left\|X_{\nu(s)}^{(i),a} - X_s^{(i)}\right\|^2 \\
&\hspace{2.5cm}+ \left\|K_{\nu_{+}(s)}-K_s\right\|^2\|g\|_{\text{Lip}}^2\sum_{i=1}^M \left\|X_s^{\text{ref}}-X_s^{(i)}\right\|^2\\
&\hspace{2.5cm} + \left\|K_{\nu_{+}(s)}\right\|^2\|g\|_{\text{Lip}}^2\sum_{i=1}^M\left\|X_{\nu_{+}(s)}^{(i),f} - X_s^{(i)}\right\|^2{\rm d}s\\
&\quad + M \left\|\int_0^{\eta(t)}\left(K_{\nu_{+}(s)}-K_s\right)C^{\frac{1}{2}}{\rm d}V_s\right\|^2\\
&\quad + \sum_{i=1}^M \left\|\int_0^{\eta(t)}\left(K_{\nu_{+}(s)}-K_s\right)C^{\frac{1}{2}}{\rm d}V_s^{(i)}\right\|^2 + \sum_{i=1}^M \left\|X_{\eta(t)}^{(i)} - X_t^{(i)}\right\|^2.
\end{aligned}
\end{equation}
Using Lemma \ref{estimatesKGain} together with
\begin{equation}\label{EstXFXA}
\begin{aligned}
\sum_{i=1}^M \left\|X_{\nu_{+}(s)}^{(i),f} - X_s^{(i)}\right\|^2 &\lesssim \sum_{i=1}^M \left\|X_{\nu(s)}^{(i),a} - X_s^{(i)}\right\|^2 \\
&\quad+ h^2\tilde{C}_f\left(M+\sum_{i=1}^M\left\|X_{\nu(s)}^{(i),a}\right\|^2\right) + \sum_{i=1}^M\left\|Q^{\frac{1}{2}}W_{\nu_{+}(s)}^{(i)}\right\|^2
\end{aligned}
\end{equation}
yields by definition of the stopping times an estimate of the form
\newpage
\begin{equation*}
\begin{aligned}
&\mathds{E}\left[\sup_{t \in[0,T \wedge \tau_n \wedge \tau_n^{h} \wedge \tau_n^{\text{ref}}]} \sum_{i=1}^M \left\|X_{\nu(t)}^{(i),a}-X_t^{(i)}\right\|^2\right]\\
&\lesssim \mathds{E}\left[\sum_{i=1}^M\left\|X_0^{(i),a} - X_0^{(i)}\right\|^2\right]\\
&\quad + T\mathds{E}\left[\int_0^{T \wedge \tau_n \wedge \tau_n^{h} \wedge \tau_n^{\text{ref}}} L(n, m_n)\sup_{r \in [0, s \wedge \tau_n \wedge \tau_n^{h} \wedge \tau_n^{\text{ref}}]}\sum_{i=1}^M \left\|X_{\nu(r)}^{(i),a} - X_r^{(i)}\right\|^2\right. \\
&\hspace{2.5cm} + R^{(1)}(n,m_n)\left(h^2\tilde{C}_f\left(M+\sum_{i=1}^M\left\|X_{\nu(s)}^{(i),a}\right\|^2\right)+\sum_{i=1}^M\left\|Q^{\frac{1}{2}}W_{\nu_{+}(s)}^{(i)}\right\|^2\right)\\
&\hspace{2.5cm}\left. + h^2R^{(2)}(n,m_n){\rm d}s\right]\\
&\quad + M\mathds{E}\left[\sup_{t \in[0,T \wedge \tau_n \wedge \tau_n^{h} \wedge \tau_n^{\text{ref}}]}\left\|\int_0^{\eta(t)}\left(K_{\nu_{+}(s)}-K_s\right)C^{\frac{1}{2}}{\rm d}V_s\right\|^2\right]\\
&\quad + \sum_{i=1}^M \mathds{E}\left[\sup_{t \in[0,T \wedge \tau_n \wedge \tau_n^{h} \wedge \tau_n^{\text{ref}}]}\left\|\int_0^{\eta(t)}\left(K_{\nu_{+}(s)}-K_s\right)C^{\frac{1}{2}}{\rm d}V_s^{(i)}\right\|^2\right]\\
&\quad + \sum_{i=1}^M \mathds{E}\left[\sup_{t \in[0,T \wedge \tau_n \wedge \tau_n^{h} \wedge \tau_n^{\text{ref}}]}\left\|X_{\eta(t)}^{(i)} - X_t^{(i)}\right\|^2\right].
\end{aligned}
\end{equation*}
By Theorem \ref{WellposedEnKFs} we have that
\[\mathds{E}\left[\int_0^{T \wedge \tau_n^{h}}\sum_{i=1}^M \left\|X_{\nu(s)}^{(i),a}\right\|^2\right]\]
is bounded. Furthermore by the Burkholder-Davis-Gundy inequality, estimate \eqref{diffGains} and \eqref{EstXFXA} we obtain
\begin{equation*}
\begin{aligned}
&\mathds{E}\left[\sup_{t \in[0,T \wedge \tau_n \wedge \tau_n^{h} \wedge \tau_n^{\text{ref}}]}\left\|\int_0^{\eta(t)}\left(K_{\nu_{+}(s)}-K_s\right)C^{\frac{1}{2}}{\rm d}V_s\right\|^2\right]\\
&\leq C_{BDG} \mathds{E}\left[\int_0^{T \wedge \tau_n \wedge \tau_n^{h} \wedge \tau_n^{\text{ref}}}\left\|K_{\nu_{+}(s)}-K_s\right\|^2\text{tr}(C){\rm d}s\right]\\
&\lesssim n\left(Th^2n^3 + \int_0^T \mathds{E}\left[\sup_{r \in[0,s \wedge \tau_n \wedge \tau_n^{h} \wedge \tau_n^{\text{ref}}]}\sum_{i=1}^M\left\|X_{\nu(r)}^{(i),a}-X_r^{(i)}\right\|^2\right]{\rm d}s \right.\\
&\qquad \left.+ \mathds{E}\left[\int_0^{T \wedge \tau_n^{h}} h^2\tilde{C}_f\left(M+\sum_{i=1}^M\left\|X_{\nu(s)}^{(i),a}\right\|^2\right){\rm d}s\right] + TMh\|Q\|\right)
\end{aligned}
\end{equation*}
where in the last line we employ Theorem \ref{WellposedEnKFs}. Finally estimate
\begin{align*}
&\left\|X_{\eta(t)}^{(i)} - X_t^{(i)}\right\|^2\notag\\
&\lesssim (t-\eta(t)) \int_{\eta(t)}^t \tilde{C}_f\left(1+\left\|X_s^{(i)}\right\|^2\right)+\left\|K_s\right\|^2\|g\|_{\text{Lip}}^2\left\|X_s^{\text{ref}}-X_s^{(i)}\right\|^2{\rm d}s \notag\\
&\quad + \left\|Q^{\frac{1}{2}}\left(W_t^{(i)}-W_{\eta(t)}^{(i)}\right)\right\|^2 + \left\|\int_{\eta(t)}^t K_sC^{\frac{1}{2}}{\rm d}V_s\right\|^2 + \left\|\int_{\eta(t)}^t K_sC^{\frac{1}{2}}{\rm d}V_s^{(i)}\right\|^2
\end{align*}
which together with Theorem \ref{WellposedEnKBF} finally yields
\begin{equation*}
\begin{aligned}
&\mathds{E}\left[\sup_{t \in[0,T \wedge \tau_n \wedge \tau_n^{h} \wedge \tau_n^{\text{ref}}]}\sum_{i=1}^M \left\|X_{\nu(t)}^{(i),a}-X_t^{(i)}\right\|^2\right]\\
&\lesssim \mathds{E}\left[\sum_{i=1}^M \left\|X_0^{(i),a}-X_0^{(i)}\right\|^2\right]\\
&\quad + C^{(1)}(T,M,n,m_n)\int_0^T\mathds{E}\left[\sup_{r \in[0,s \wedge \tau_n \wedge \tau_n^{h} \wedge \tau_n^{\text{ref}}]}\sum_{i=1}^M\left\|X_{\nu(r)}^{(i),a}-X_r^{(i)}\right\|^2\right] {\rm d}s \\
&\quad+ hC^{(2)}(T,M,n,m_n,h)
\end{aligned}
\end{equation*}
where
\begin{align*}
C^{(1)}(T,M,n,m_n) &:= \tilde{C}^{(1)}\left(T(1 + nm_n)+nM\right),\\
C^{(2)}(T,M,n,m_n,h)&:= \tilde{C}^{(2)}\left(h\left((Tn(1+n)+M)(1+MT)\right.\right.\\
&\hspace{2cm}\left.+T^2n^4m_n+1+M+n^2m_n\right)\\
&\hspace{1.5cm}\left.+(Tn(1+n)+M)TM+M(1+n^2)\right)
\end{align*}
and $\tilde{C}^{(1)}, \tilde{C}^{(2)}$ are positive constants determined by $\|f\|_{\rm Lip}^2$, $\|g\|_{\rm Lip}^2$, $\|Q\|$, ${\rm tr}(C)$, $\kappa_{\rm diff}^2$, $\kappa_K^2$ and the constants in Theorem \ref{WellposedEnKBF} and Theorem \ref{WellposedEnKFs}. Thus applying a standard Gronwall argument and imposing
\[\mathds{E}\left[\sum_{i=1}^M \left\|X_0^{(i),a}-X_0^{(i)}\right\|^2\right]\lesssim h\]
yields the claim.
\end{proof}
\newpage
\noindent
A similar result holds in the case of the ESRF converging to the ETKBF:
\begin{theorem}\label{mainResult2}
Let $\left(W_t^{(i)}\right), \left(V_t^{(i)}\right), i=1,...,M,$ be independent Brownian motions and denote $W_k^{(i)} := W_{t_k}^{(i)} -W_{t_{k-1}}^{(i)}$ and $V_k^{(i)} := V_{t_k}^{(i)} - V_{t_{k-1}}^{(i)}$. Let $\left(X_k^{(i),f}, X_k^{(i),a}\right), i=1,...,M$, denote the two ESRF algorithms \eqref{ESRFF}-\eqref{ESRFA} with respect to $\left(W_k^{(i)}\right), \left(V_k^{(i)}\right)$ and $\left(X_t^{(i)}\right), i=1,...,M$, a strong solution of the ETKBF-SDE \eqref{ETKBF} with respect to $W_t^{(i)}$ and $V_t^{(i)}$. Then $X_t^{(i)}$ forms a continuous time limit of $X_k^{(i),f}, X_k^{(i),a}$ in the sense that if
\begin{equation}
\mathbb{E}\left[\sum_{i=1}^M \left\|X_0^{(i),a}-X_0^{(i)}\right\|^2\right]\lesssim h,
\end{equation}
then there exists a constant $\mathcal{C}= \mathcal{C}(T,M,n,m_n)$ such that
\begin{equation}
\mathds{E}\left[\sup_{t \in [0, T \wedge \tau_n \wedge \tau_n^{h} \wedge \tau_n^{{\rm ref}}]}\sum_{i=1}^M\left\|X_{\nu(t)}^{(i),a} - X_t^{(i)}\right\|^2\right] \leq \mathcal{C}h
\end{equation}
where for the stopping time
\begin{equation*}
\tau_n^{{\rm ref}} := \inf\left\{ t > 0: \sum_{i=1}^M \left\|X_t^{{\rm ref}}-X_t^{(i)}\right\|^2 > m_n\right\}
\end{equation*}
the value $m_n$ can be chosen in such a way that
\begin{equation}
\limsup_{n \rightarrow \infty} \mathds{P}\left[\tau_n^{{\rm ref}} \leq T \wedge \tau_n\right] =0
\end{equation}
for every $T>0$.
\end{theorem}
\begin{proof}
Again choose $m_n$ as done in the proof of Theorem \ref{mainResult1}. Via the Cauchy-Schwarz inequality, the ensemble-mean squared norm of \eqref{decompDiffESRF} admits the following estimate:
\begin{equation*}
\begin{aligned}
&\sum_{i=1}^M \left\|X_{\nu(t)}^{(i),a} - X_t^{(i)}\right\|^2\\
&\lesssim \sum_{i=1}^M \left\|X_0^{(i),a} - X_0^{(i)}\right\|^2\\
&\quad + \eta(t) \int_0^{\eta(t)} \|f\|_{\text{Lip}}^2\sum_{i=1}^M\left\|X_{\nu(s)}^{(i),a} - X_s^{(i)}\right\|^2 \\
&\hspace{2.5cm}+ \left(\left\|\hat{K}_{\nu_{+}(s)}-\frac{1}{2}K_s\right\|^2+\left\|K_{\nu_{+}(s)} - \hat{K}_{\nu_{+}(s)}-\frac{1}{2}K_s\right\|^2\right)\\
&\hspace{3cm}\times\|g\|_{\text{Lip}}^2\sum_{i=1}^M \left\|X_s^{\text{ref}}-X_s^{(i)}\right\|^2\\
&\hspace{2.5cm} + \left(\left\|\hat{K}_{\nu_{+}(s)}\right\|^2+\left\|K_{\nu_{+}(s)}-\hat{K}_{\nu_{+}(s)}\right\|^2\right)\\
&\hspace{3cm}\times\|g\|_{\text{Lip}}^2\sum_{i=1}^M\left\|X_{\nu_{+}(s)}^{(i),f} - X_s^{(i)}\right\|^2{\rm d}s\\
&\quad + M \left\|\int_0^{\eta(t)}\left(K_{\nu_{+}(s)}-\frac{1}{M-1}E_s\mathcal{G}_s^TC^{-1}\right)C^{\frac{1}{2}}{\rm d}V_s\right\|^2\\
&\quad + \sum_{i=1}^M \left\|\sum_{k=1}^{\nu(t)} \mathcal{R}_k^{(i)}\right\|^2+ \sum_{i=1}^M \left\|X_{\eta(t)}^{(i)} - X_t^{(i)}\right\|^2.
\end{aligned}
\end{equation*}
Using the identity
\[ e^{-tP} - {\rm Id} = - \int_0^t e^{-sP}P{\rm d}s\]
for symmetric positive semidefinite matrix $P$, we estimate
\begin{equation*}
\left\|\mathcal{R}_k^{(i)}\right\|^2 \lesssim h^4\text{tr}\left(P_k^{f}\right)^5
\end{equation*}
which yields
\begin{equation*}
\sum_{i=1}^M \left\|\sum_{k=1}^{\nu(t)} \mathcal{R}_k^{(i)}\right\|^2 \lesssim M T h^3 \text{tr}\left(P_k^{f}\right)^5.
\end{equation*}
The rest of the proof is analogous to the proof of Theorem \ref{mainResult1}.
\end{proof}

\section{Discussion and Conclusion}
\noindent
The Ensemble Kalman-Bucy Filter and the Ensemble Transform Kalman-Bucy Filter form two popular continuous-time filtering algorithms proposed in the literature which prove to work as an effective tool in the property analysis for discrete-time ensemble-based Kalman-type filters. A crucial question therefore is whether both EnKBF and ETKBF are naturally related to their discrete-time counterparts in the sense that they can be derived via a continuous time limit analysis. In the general finite dimensional setting of nonlinear, unbounded signal and observations, this paper provides an affirmative answer in that we are able to rigorously show that both EnKBF and ETKBF are the results of taking the continuous time limit of the well-known filtering algorithms EnKF and ESRF, respectively. Herewith we close the current gap in the literature which provides only formal argumentations. Most notably, this paper provides a general limiting result in that we do not impose restrictive assumptions on the setting, or on the algorithmic structure as done in previous works of ours. Second of all, observe that the ETKBF forms the continuous time limit of the EAKF, ETKF and the unperturbed filter in the sense specified and as shown in Theorem \ref{mainResult2} which suggests that it forms the universal continuous time limit of the class of ESRF algorithms.\\
Observe that a core element in the analysis is to control the empirical covariance matrices. In our previous works we handled this in the following way: assuming boundedness of the observation operator $g$ in \cite{lange2019}, terms involving the empirical covariance matrix were of suitable lower order such that they were absorbed by the Gronwall argument; and by replacing the model noise by suitable deterministic forecast perturbations as in \cite{lange2019b} inspired by \cite{deWiljes2018}, we were even able to derive a priori upper bounds. In the setting considered in this paper, this is no longer the case hence we employ a suitable stopping time argument. Together with the derivation of the continuous time limit, this enabled us to show boundedness results for both the discrete-time and continuous-time algorithms which hold in probability and can be interpreted as well-posedness and accuracy results. Comparable results have appeared in the literature before (as discussed in the introduction) and we have been inspired by \cite{kelly2014} for the EnK(B)F, as well as \cite{deWiljes2018} for the ETKBF. Note that the results in  \cite{kelly2014} and \cite{deWiljes2018} are derived without stopping since both work in the regime of full observations, additionally assuming the observation covariance matrix to be diagonal. Additionally in \cite{deWiljes2018}, the authors replace the model noise by suitable perturbations fitting the Riccati equation of the covariance matrix hence enforcing the required boundedness.\\
Observe furthermore that our results hold only locally in time. Indeed, the bounding constants depend exponentially on the time horizon $T$, therefore explode for $T \rightarrow \infty$. Note that a similar behavior holds for the well-posedness results in \cite{kelly2014}. Since this is a consequence of the use of Gronwall arguments, it would be interesting to investigate whether alternative reasonings apply in this setting. Note, however, that in the partially observed case as we consider here, global-in-time results would necessitate further structural assumptions on the underlying setting including observability and controllability conditions. Possible strategies in this direction have been initiated in the continuous-time linearly observed case in \cite{delMoral2018} for the EnKBF with linear forecast drift, and in \cite{delMoral2017} for an extended EnKBF using concentration inequalities.

\bigskip 
\noindent 
{\bf Acknowledgements} The research of Theresa Lange has been partially funded by Deutsche Forschungsgemeinschaft (DFG) - SFB1294/1 - 318763901. The author would further like to thank the referees for their detailed and constructive comments.

\appendix
\section{Frequently used Kalman gain estimates}\label{appLemma}
\noindent
For convenience we repeat the statement of Lemma \ref{estimatesKGain} here:
\begin{lemma}\label{estimatesK}
The operator norms of the discrete-time and continuous-time Kalman gains respectively satisfy
\begin{align}
\left\|K_{\nu(t)}\right\| &\leq \kappa_K{\rm tr}\left(P_{\nu(t)}^{f}\right) \label{KDiscr}\\
\left\|\hat{K}_{\nu(t)}\right\| &\leq \frac{1}{2}\kappa_K{\rm tr}\left(P_{\nu(t)}^{f}\right)\label{HatKDiscr}\\
\left\|K_t\right\| &\leq \kappa_K{\rm tr}\left(P_t\right)\label{KCont}
\end{align}
where $\kappa_K := \|g\|_{\text{Lip}}\left\|C^{-1}\right\|$. Further let $\Delta_t$ denote 
\[ \left\|K_{\nu(t)} - K_t\right\|, \left\|\hat{K}_{\nu(t)} - \frac{1}{2}K_t\right\|, \text{ or } \left\|K_{\nu(t)}-\hat{K}_{\nu(t)}-\frac{1}{2}K_t\right\|,\]
then $\Delta_t$ satisfies the following estimate
\begin{equation}\label{DiffGains}
\Delta_t \leq \kappa_{\text{diff}}\left( h{\rm tr}\left(P_{\nu(t)}^{f}\right)^2 + \left({\rm tr}\left(P_{\nu(t)}^{f}\right)^{\frac{1}{2}}+{\rm tr}(P_t)^{\frac{1}{2}}\right)\left(\sum_{i=1}^M \left\|X_{\nu(t)}^{(i),f} - X_t^{(i)}\right\|^2\right)^{\frac{1}{2}}\right)
\end{equation}
for a constant $\kappa_{\text{diff}}$ different for each case of $\Delta_t$, but independent of $h$.
\end{lemma}
\begin{proof}
A key observation in this proof is that with the notation
\[ \tilde{\mathcal{G}}_{\cdot} := \left[g\left(X_{\cdot}^{(i)}\right) - g\left(\bar{x}_{\cdot}\right)\right]_{i=1}^M\]
it holds
\begin{equation}\label{ETildeG}
E_{\cdot}\mathcal{G}_{\cdot}^T = E_{\cdot}\tilde{\mathcal{G}}_{\cdot}^T.
\end{equation}
Consequently we may estimate
\begin{equation}\label{EG}
\begin{aligned}
\left\|E_{\nu(t)}^{f}\left(\mathcal{G}_{\nu(t)}^{f}\right)^T\right\| &= \left\|E_{\nu(t)}^{f}\left(\tilde{\mathcal{G}}_{\nu(t)}^{f}\right)^T\right\|\\
&= \left\|\sum_{i=1}^M \left(X_{\nu(t)}^{(i),f} - \bar{x}_{\nu(t)}^{f}\right)\left(g\left(X_{\nu(t)}^{(i),f}\right)- g\left(\bar{x}_{\nu(t)}^{f}\right)^T\right)\right\|\\
&\leq \sum_{i=1}^M \left\|X_{\nu(t)}^{(i),f}-\bar{x}_{\nu(t)}^{f}\right\|\left\|g\left(X_{\nu(t)}^{(i),f}\right)-g\left(\bar{x}_{\nu(t)}^{f}\right)\right\|\\
&\leq \|g\|_{\text{Lip}}\sum_{i=1}^M \left\|X_{\nu(t)}^{(i),f} - \bar{x}_{\nu(t)}^{f}\right\|^2 = (M-1)\|g\|_{\text{Lip}}{\rm tr}\left(P_{\nu(t)}^{f}\right)
\end{aligned}
\end{equation}
which yields \eqref{HatKDiscr} in case of the EAKF/ETKF, as well as \eqref{KDiscr} due to
\begin{equation*}
\left(C+\frac{h}{M-1}\mathcal{G}_{\nu(t)}^{f}\left(\mathcal{G}_{\nu(t)}^{f}\right)^T\right)^{-1} \leq C^{-1}.
\end{equation*}
Similar to \eqref{EG} it holds
\begin{equation*}
\left\|E_t\mathcal{G}_t^T\right\| \leq (M-1)\|g\|_{\text{Lip}}\text{tr}(P_t)
\end{equation*}
which yields \eqref{KCont}. Finally observe that since
\begin{equation*}
\left(C+\frac{h}{M-1}\mathcal{G}_{\nu(t)}^{f}\left(\mathcal{G}_{\nu(t)}^{f}\right)^T\right)^{\frac{1}{2}} \geq C^{\frac{1}{2}}
\end{equation*}
we may estimate
\[ \tilde{C}_{\nu(t)} := \left(C+\frac{h}{M-1}\mathcal{G}_{\nu(t)}^{f}\left(\mathcal{G}_{\nu(t)}^{f}\right)^T\right)^{-\frac{1}{2}}\left(C^{\frac{1}{2}}+\left(C+\frac{h}{M-1}\mathcal{G}_{\nu(t)}^{f}\left(\mathcal{G}_{\nu(t)}^{f}\right)^T\right)^{\frac{1}{2}}\right)^{-1}\]
by
\begin{equation*}
\left\|\tilde{C}_{\nu(t)}\right\| \leq \frac{1}{2}\left\|C^{-\frac{1}{2}}\right\|^2 = \frac{1}{2}\left\|C^{-1}\right\|
\end{equation*}
which gives \eqref{HatKDiscr} in case of the unperturbed filter by \cite{whitaker2002}.\\
On $\Delta_t := \left\|K_{\nu(t)} - K_t\right\|$:
\begin{align*}
&K_{\nu(t)} - K_t\notag\\
 &= \frac{1}{M-1}\left(E_{\nu(t)}^{f}\left(\mathcal{G}_{\nu(t)}^{f}\right)^T\left(C+\frac{h}{M-1}\mathcal{G}_{\nu(t)}^{f}\left(\mathcal{G}_{\nu(t)}^{f}\right)^T\right)^{-1} - E_t\mathcal{G}_t^TC^{-1}\right)\notag\\
&= \frac{1}{M-1}\left( E_{\nu(t)}^{f}\left(\mathcal{G}_{\nu(t)}^{f}\right)^T\left( \left( C+\frac{h}{M-1}\mathcal{G}_{\nu(t)}^{f}\left(\mathcal{G}_{\nu(t)}^{f}\right)^T\right)^{-1} - C^{-1}\right) \right.\notag\\
&\hspace{2cm}\left.+ \left(E_{\nu(t)}^{f}\left(\mathcal{G}_{\nu(t)}^{f}\right)^T - E_t\mathcal{G}_t^T\right)C^{-1}\right).
\end{align*}
For the first summand we observe that
\begin{align*}
&\left(C+\frac{h}{M-1}\mathcal{G}_{\nu(t)}^{f}\left(\mathcal{G}_{\nu(t)}^{f}\right)^T\right)^{-1} - C^{-1}\notag\\
&= C^{-\frac{1}{2}}\left( \left({\rm Id} + \frac{h}{M-1}C^{-\frac{1}{2}}\mathcal{G}_{\nu(t)}^{f}\left(\mathcal{G}_{\nu(t)}^{f}\right)^TC^{-\frac{1}{2}}\right)^{-1} - {\rm Id}\right)C^{-\frac{1}{2}}\\
&= C^{-\frac{1}{2}}\left( \sum_{n=1}^{\infty} \left(-\frac{h}{M-1}C^{-\frac{1}{2}}\mathcal{G}_{\nu(t)}^{f}\left(\mathcal{G}_{\nu(t)}^{f}\right)^TC^{-\frac{1}{2}}\right)^n\right)C^{-\frac{1}{2}}\\
&= -\frac{h}{M-1} C^{-1}\mathcal{G}_{\nu(t)}^{f} \left({\rm Id} + \frac{h}{M-1} \left(\mathcal{G}_{\nu(t)}^{f}\right)^TC^{-1}\mathcal{G}_{\nu(t)}^{f}\right)^{-1}\left(\mathcal{G}_{\nu(t)}^{f}\right)^TC^{-1}.
\end{align*}
Thus since
\[ \left({\rm Id} + \frac{h}{M-1}\left(\mathcal{G}_{\nu(t)}^{f}\right)^TC^{-1}\mathcal{G}_{\nu(t)}^{f}\right)^{-1} \leq {\rm Id}\]
in the sense of symmetric positive semidefinite matrices, we obtain
\begin{align*}
&\left\|\left(C+\frac{h}{M-1}\mathcal{G}_{\nu(t)}^{f}\left(\mathcal{G}_{\nu(t)}^{f}\right)^T\right)^{-1} - C^{-1}\right\| \notag\\
&\leq \frac{h}{M-1}\left\|C^{-1}\mathcal{G}_{\nu(t)}^{f}\left(\mathcal{G}_{\nu(t)}^{f}\right)^TC^{-1}\right\| \leq \frac{h}{M-1}\left\|C^{-1}\right\|^2\left\|\mathcal{G}_{\nu(t)}^{f}\left(\mathcal{G}_{\nu(t)}^{f}\right)^T\right\|
\end{align*}
where
\begin{align*}
\left\|\mathcal{G}_{\nu(t)}^{f}\left(\mathcal{G}_{\nu(t)}^{f}\right)^T\right\|
&\leq \sum_{i=1}^M \left\|g\left(X_{\nu(t)}^{(i),f}\right) - g\left(\bar{x}_{\nu(t)}^{f}\right)\right\|^2 + M\left\|\bar{g}_{\nu(t)}^{f} - g\left(\bar{x}_{\nu(t)}^{f}\right)\right\|^2\\
&\leq 2(M-1)\|g\|_{\text{Lip}}^2\text{tr}\left(P_{\nu(t)}^{f}\right).
\end{align*}
For the second summand note that with \eqref{ETildeG} it holds
\begin{equation*}
\begin{aligned}
&\frac{1}{M-1}\left\|E_{\nu(t)}^{f}\left(\mathcal{G}_{\nu(t)}^{f}\right)^T - E_t\mathcal{G}_t^T\right\|\\
&\leq \left(\frac{1}{M-1}\sum_{i=1}^M\left\|\left(X_{\nu(t)}^{(i),f} - \bar{x}_{\nu(t)}^{f}\right)-\left(X_t^{(i)}-\bar{x}_t\right)\right\|^2\right)^{\frac{1}{2}}\\
&\hspace{1cm}\times\left(\frac{1}{M-1}\sum_{i=1}^M\left\|g\left(X_{\nu(t)}^{(i),f}\right)-g\left(\bar{x}_{\nu(t)}^{f}\right)\right\|^2\right)^{\frac{1}{2}}\\
&\quad + \left(\frac{1}{M-1}\sum_{i=1}^M \left\|X_t^{(i)}-\bar{x}_t\right\|^2\right)^{\frac{1}{2}}\\
&\hspace{1cm}\times\left(\frac{1}{M-1}\sum_{i=1}^M\left\|\left(g\left(X_{\nu(t)}^{(i),f}\right)-g\left(\bar{x}_{\nu(t)}^{f}\right)\right)-\left(g\left(X_t^{(i)}\right)-g\left(\bar{x}_t\right)\right)\right\|^2\right)^{\frac{1}{2}}\\
&\leq \|g\|_{\text{Lip}}\left(\text{tr}\left(P_{\nu(t)}^{f}\right)^{\frac{1}{2}}+\text{tr}(P_t)^{\frac{1}{2}}\right)\left(\frac{1}{M-1}\sum_{i=1}^M \left\|X_{\nu(t)}^{(i),f}-X_t^{(i)}\right\|^2\right)^{\frac{1}{2}}
\end{aligned}
\end{equation*}
which in total yields the claim.\\
On $\Delta_t := \left\|\hat{K}_{\nu(t)} - \frac{1}{2}K_t\right\|$: for the EAKF/ETKF the claim immediately follows from
\begin{equation*}
\hat{K}_{\nu(t)} - \frac{1}{2}K_t = \frac{1}{2}\frac{1}{M-1}\left(E_{\nu(t)}^{f}\left(\mathcal{G}_{\nu(t)}^{f}\right)^T - E_t\mathcal{G}_t^T\right)C^{-1}
\end{equation*}
and the above analyzes. For the unperturbed filter consider the following decomposition
\begin{equation*}
\begin{aligned}
&\tilde{C}_{\nu(t)} - \frac{1}{2}C^{-1}\\
&= \tilde{C}_{\nu(t)}\left(C^{\frac{1}{2}}\left(C^{\frac{1}{2}}-\left(C+\frac{h}{M-1}\mathcal{G}_{\nu(t)}^{f}\left(\mathcal{G}_{\nu(t)}^{f}\right)^T\right)^{\frac{1}{2}}\right) - \frac{h}{M-1}\mathcal{G}_{\nu(t)}^{f}\left(\mathcal{G}_{\nu(t)}^{f}\right)^T\right)\frac{1}{2}C^{-1}.
\end{aligned}
\end{equation*}
Thus using
\begin{equation*}
\left\|C^{\frac{1}{2}}-\left(C+\frac{h}{M-1}\mathcal{G}_{\nu(t)}^{f}\left(\mathcal{G}_{\nu(t)}^{f}\right)^T\right)^{\frac{1}{2}}\right\| \lesssim \left\|\frac{h}{M-1}\mathcal{G}_{\nu(t)}^{f}\left(\mathcal{G}_{\nu(t)}^{f}\right)^T\right\|
\end{equation*}
together with the above estimates yields the claim.\\
The estimate for $\Delta_t := \left\|K_{\nu(t)}-\hat{K}_{\nu(t)}-\frac{1}{2}K_t\right\|$ follows from similar decompositions and estimates as presented above.
\end{proof}

\section{Well-posedness results}
\subsection{Proof of Theorem \ref{WellposedEnKBF}}\label{appWellposedEnKBF}
Recall from Section \ref{ContWellPosed} that it holds
\begin{equation*}
\mathds{E}\left[\sup_{t \in [0,T \wedge \xi]}\sqrt{\text{tr}(P_t)}\right] \lesssim e^{\|f\|_{\text{Lip}}T}\text{tr}(Q).
\end{equation*}
Hence by Markov's inequality we deduce
\begin{equation*}
\mathds{P}\left[\tau_n \leq T\right] \leq \frac{1}{\sqrt{n}}\mathds{E}\left[\sqrt{\text{tr}(P_{\nu\left(\tau_n \wedge T\right)})}\right] \longrightarrow 0, n \rightarrow \infty.
\end{equation*}
Using \eqref{YXref}, we obtain by It\^{o}'s formula
\begin{align*}
{\rm d}\frac{1}{2}\sum_{i=1}^M \left\|X_t^{(i)}\right\|^2&= \sum_{i=1}^M \left \langle X_t^{(i)}, f\left(X_t^{(i)}\right) + K_t\left(g\left(X_t^{\text{ref}}\right) - g\left(X_t^{(i)}\right)\right)\right \rangle{\rm d}t \\
&\hspace{1cm} + \left \langle X_t^{(i)}, Q^{\frac{1}{2}}{\rm d}W_t^{(i)} + K_tC^{\frac{1}{2}}\left({\rm d}V_t + {\rm d}V_t^{(i)}\right) \right \rangle \\
&\hspace{1cm}+ \left(\frac{1}{2}\text{tr}(Q) + \text{tr}\left(K_tCK_t^T\right)\right){\rm d}t.
\end{align*}
Thus by global Lipschitz continuity of $f$ and $g$ it holds
\begin{equation*}
{\rm d}\sum_{i=1}^M \left\|X_t^{(i)}\right\|^2 \leq \left((1 + 2C_f \sum_{i=1}^M\left\|X_t^{(i)}\right\|^2 + \mathcal{R}(t)\right){\rm d}t + \mathcal{N}_t
\end{equation*}
where together with Lemma \ref{estimatesK} it holds
\begin{align*}
\mathcal{R}(t) &:= 2M\left(C_f + \frac{1}{2}\text{tr}(Q) + \left(\kappa_K \text{tr}(P_t)\right)^2\text{tr}(C)\right) \\
&\quad+ \left(\kappa_K \text{tr}(P_t)\right)^2\|g\|_{\text{Lip}}^2\sum_{i=1}^M \left\|X_t^{\text{ref}}- X_t^{(i)}\right\|^2,\\
\mathcal{N}_t &:= \sum_{i=1}^M \left \langle X_t^{(i)}, Q^{\frac{1}{2}}{\rm d}W_t^{(i)}\right \rangle + \left\langle X_t^{(i)}, K_tC^{\frac{1}{2}}{\rm d}V_t\right \rangle + \left\langle X_t^{(i)}, K_tC^{\frac{1}{2}}{\rm d}V_t^{(i)}\right \rangle.
\end{align*}
For each $t \in [0, T \wedge \tau_n]$ it holds $\mathcal{R}(t) \leq \mathcal{R}^{n}$ where
\begin{equation*}
\mathcal{R}^{n} := M\left(C_f + \frac{1}{2}\text{tr}(Q) + (\kappa_K n)^2\text{tr}(C)\right)+ (\kappa_K n)^2 \sup_{t \in [0, T \wedge \tau_n]} \sum_{i=1}^M \left\|X_s^{\text{ref}}-X_s^{(i)}\right\|^2
\end{equation*}
which yields
\begin{align*}
\sup_{t \in [0, T \wedge \tau_n]} \sum_{i=1}^M \left\|X_t^{(i)}\right\|^2 &\leq \sum_{i=1}^M \left\|X_0^{(i)}\right\|^2 + \int_0^{T \wedge \tau_n} (1+2C_f) \sup_{r \in [0,s \wedge \tau_n]}\sum_{i=1}^M \left\|X_r^{(i)}\right\|^2 {\rm d}s\notag\\
&\hspace{0.5cm} + (T \wedge \tau_n)\mathcal{R}^{(n)} + \sup_{t \in [0, T\wedge \tau_n]} \int_0^t {\rm d} \mathcal{N}_s.
\end{align*}
By the Burkholder-Davis-Gundy inequality we obtain
\begin{align*}
&\mathds{E}\left[\sup_{t \in [0, T \wedge \tau_n]} \int_0^t {\rm d}\mathcal{N}_s\right]\\
&\leq C_{BDG}\left( \sum_{i=1}^M \mathds{E}\left[\left( \int_0^{T \wedge \tau_n}\|Q\|\left\|X_s^{(i)}\right\|^2{\rm d}s\right)^{\frac{1}{2}}\right]\right.\\
&\hspace{2.5cm}\left. + 2\mathds{E}\left[\left(\int_0^{T \wedge \tau_n} \|C\|\left(\kappa_K \text{tr}(P_s)\right)^2\left\|X_s^{(i)}\right\|^2{\rm d}s\right)^{\frac{1}{2}}\right]\right)\\
&\leq \frac{C_{BDG}}{2}\left(3M + \left(\|Q\| + (\kappa_K n)^2\|C\|\right)\int_0^T \mathds{E}\left[\sum_{i=1}^M \left\|X_s^{(i)}\right\|^2\right]{\rm d}s\right).
\end{align*}
Using Theorem \ref{AccCont} and applying a Gronwall argument then yields the claim.

\subsection{Proof of Theorem \ref{WellposedEnKFs}}\label{appWellposedEnKFs}
\noindent
In order to keep the paper concise, we will only show the claims for the EnKF. The proof for the ESRF algorithms is analogous.\\
\noindent
On the claim on the stopping time $\tau_n^{h}$: similar to the continuous-time case, if we can show
\begin{equation*}
\mathds{E}\left[\sup_{t \in [0,T \wedge \xi]}\left(\text{tr}\left(P_{\nu(t)}^{f}\right)\right)^{p}\right] \leq \mathcal{C}
\end{equation*}
up to explosion time $\xi$ for some constant $\mathcal{C}>0$ and some $p>0$, the claim follows by Markov's inequality. The trace of the empirical covariance matrix is determined by the following recursions: for the forecast covariance matrix
\begin{align*}
&\text{tr}\left(P_k^{f}\right)\\
&= \frac{1}{M-1} \sum_{i=1}^M \left\|X_{k-1}^{(i),a} - \bar{x}_{k-1}^{a}\right\|^2\\
&\hspace{2.3cm} + 2h\left \langle X_{k-1}^{(i),a} - \bar{x}_{k-1}^{a}, f\left(X_{k-1}^{(i),a}\right) - \bar{f}_{k-1}^{a}\right \rangle + h^2\left\|f\left(X_{k-1}^{(i),a}\right) - \bar{f}_{k-1}^{a}\right\|^2\\
&\hspace{2.3cm} + 2\left \langle X_{k-1}^{(i),a} - \bar{x}_{k-1}^{a} + h\left(f\left(X_{k-1}^{(i),a}\right) - \bar{f}_{k-1}^{a}\right), Q^{\frac{1}{2}}\left(W_k^{(i)} - \bar{w}_k\right) \right \rangle \\
&\hspace{2.3cm}+ \left\|Q^{\frac{1}{2}}\left(W_k^{(i)} - \bar{w}_k\right)\right\|^2
\end{align*}
(note that this holds true for all algorithms presented in Section \ref{algos}) and the analysis covariance matrix
\begin{align*}
&\text{tr}\left(P_k^{a}\right)\\
&= \frac{1}{M-1} \sum_{i=1}^M \left\|X_k^{(i),f} - \bar{x}_k^{f}\right\|^2\\
&\hspace{2.3cm} -2h\left \langle X_k^{(i),f} - \bar{x}_k^{f}, K_k\left(g\left(X_k^{(i),f}\right) - \bar{g}_k^{f}\right)\right \rangle\\
&\hspace{2.3cm} + h^2\left\|K_k\left(g\left(X_k^{(i),f}\right) - \bar{g}_k^{f}\right)\right\|^2\\
&\hspace{2.3cm} + 2\left\langle X_k^{(i),f} - \bar{x}_k^{f} - hK_k\left(g\left(X_k^{(i),f}\right) - \bar{g}_k^{f}\right), K_kC^{\frac{1}{2}}\left(V_k^{(i)} - \bar{v}_k\right)\right \rangle\\
&\hspace{2.3cm} + \left\|K_kC^{\frac{1}{2}}\left(V_k^{(i)} - \bar{v}_k\right)\right\|^2.
\end{align*}
Observe the following identities: we can rewrite
\begin{equation*}
\sum_{i=1}^M \left \langle X_k^{(i),f} - \bar{x}_k^{f}, K_k\left(g\left(X_k^{(i),f}\right) - \bar{g}_k^{f}\right)\right \rangle = \text{tr}\left(E_k^{f}\left(\mathcal{G}_k^{f}\right)^TK_k^T\right)
\end{equation*}
as well as
\begin{align*}
&\frac{h^2}{M-1}\sum_{i=1}^M \left\|K_k\left(g\left(X_k^{(i),f}\right) - \bar{g}_k^{f}\right)\right\|^2= \frac{h^2}{M-1} \text{tr}\left(K_k\mathcal{G}_k^{f}\left(\mathcal{G}_k^{f}\right)^TK_k^T\right)\\
&= h\left( \text{tr}\left(K_k\left(C + \frac{h}{M-1}\mathcal{G}_k^{f}\left(\mathcal{G}_k^{f}\right)^T\right)K_k^T\right) - \text{tr}\left(K_kCK_k^T\right)\right)\\
&= h\left(\frac{1}{M-1}\text{tr}\left(K_k\mathcal{G}_k^{f}\left(E_k^{f}\right)^T\right) - \text{tr}\left(K_kCK_k^T\right)\right).
\end{align*}
Hence by Lipschitz-continuity of $f$ we obtain the following recursive inequalities:
\begin{equation*}
\begin{aligned}
\text{tr}\left(P_k^{f}\right) &\leq \left(1+h\mathcal{C}(h)\right)\text{tr}\left(P_{k-1}^{a}\right) + N_k^{f} + \text{tr}\left(Q^{\frac{1}{2}}\mathds{W}_kQ^{\frac{1}{2}}\right),\\
\text{tr}\left(P_k^{a}\right) &\leq \text{tr}\left(P_k^{f}\right) + N_k^{a} + \text{tr}\left(K_kC^{\frac{1}{2}}\left(\mathds{V}_k-h{\rm Id}\right)C^{\frac{1}{2}}K_k^T\right)
\end{aligned}
\end{equation*}
for a constant $\mathcal{C}(h) = \mathcal{C}(h,f)>0$ and where
\begin{align*}
N_k^{f} &:= \frac{2}{M-1} \sum_{i=1}^M \left\langle X_{k-1}^{(i),a} - \bar{x}_{k-1}^{a} + h\left(f\left(X_{k-1}^{(i),a}\right)-\bar{f}_{k-1}^{a}\right), Q^{\frac{1}{2}}\left(W_k^{(i)}-\bar{w}_k\right)\right \rangle,\\
\mathbb{W}_k &:= \frac{1}{M-1}\sum_{i=1}^M \left(W_k^{(i)} - \bar{w}_k\right)\left(W_k^{(i)}-\bar{w}_k\right)^T
\end{align*}
as well as
\begin{align*}
N_k^{a} &:= \frac{2}{M-1} \sum_{i=1}^M \left \langle X_k^{(i),f} - \bar{x}_k^{f} - hK_k\left(g\left(X_k^{(i),f}\right) - \bar{g}_k^{f}\right), K_kC^{\frac{1}{2}}\left(V_k^{(i)}-\bar{v}_k\right)\right\rangle,\\
\mathds{V}_k &:= \frac{1}{M-1}\sum_{i=1}^M \left(V_k^{(i)} - \bar{v}_k\right)\left(V_k^{(i)}-\bar{v}_k\right)^T.
\end{align*}
Now consider the discrete-time process $Z = \left(Z_k\right)_{0 \leq k \leq L}$ defined by
\begin{equation*}
Z_{2k} := \text{tr}\left(P_k^{a}\right),\qquad Z_{2k+1} := \text{tr}\left(P_{k+1}^{f}\right)
\end{equation*}
and the filtration $\left(\mathcal{F}_k\right)_{0\leq k \leq L}$ given by
\begin{equation*}
\begin{aligned}
\mathcal{F}_0 &:= \sigma\left(X_0^{(i),a}\right)\vee \mathcal{Y}_T,\\
\mathcal{F}_{2k} &:= \mathcal{F}_{2k-1} \vee \sigma\left(V_k^{(i)}, i=1,...,M\right),\\
\mathcal{F}_{2k+1} &:= \mathcal{F}_{2k}\vee \sigma\left(W_{k+1}^{(i)}, i=1,...,M\right)
\end{aligned}
\end{equation*}
(where $\vee$ denotes the operation of forming the smallest $\sigma$-algebra containing all listed $\sigma$-algebras), then $Z$ is an $\left(\mathcal{F}_k\right)$-adapted process and satisfies
\begin{equation*}
Z_k \leq F_k + \mathcal{N}_k + \sum_{l=0}^{k-1}G_lZ_l
\end{equation*}
where
\begin{equation*}
\begin{aligned}
F_k &:= \left(1+h\mathcal{C}(h)\right)\text{tr}\left(P_0^{a}\right) + \sum_{l=1}^{\lfloor \frac{k+1}{2} \rfloor} \text{tr}\left(Q^{\frac{1}{2}}\mathds{W}_lQ^{\frac{1}{2}}\right),\\
G_k &:= h \mathcal{C}(h)
\end{aligned}
\end{equation*}
and a process $\mathcal{N}$ defined by
\begin{equation*}
\begin{aligned}
\mathcal{N}_{2k} &:= \sum_{l=1}^k N_l^{f} + N_l^{a} + \text{tr}\left(K_lC^{\frac{1}{2}}\left(\mathds{V}_l-h{\rm Id}\right)C^{\frac{1}{2}}K_l^T\right),\\
\mathcal{N}_{2k+1} &:= \sum_{l=1}^{k+1} N_l^{f} + \sum_{l=1}^k \left(1+h\mathcal{C}(h)\right)\left(N_l^{a} + \text{tr}\left(K_lC^{\frac{1}{2}}\left(\mathds{V}_l-h{\rm Id}\right)C^{\frac{1}{2}}K_l^T\right)\right)
\end{aligned}
\end{equation*}
\noindent
Observe that since
\begin{equation*}
V_k^{(i)} - \bar{v}_k \sim \mathcal{N}\left(0, \frac{M-1}{M}h{\rm Id}\right)
\end{equation*}
and $\left(V_k^{(i)}\right)_{i=1,...,M}$ are chosen independently of the forecast ensemble at time $k$, the conditional expectation, with respect to the $\sigma$-algebra generated by the forecast ensemble, of $\mathds{V}_k$ equals $h{\rm Id}$. Hence $\mathcal{N}$ is an $\left(\mathcal{F}_k\right)$-martingale and applying the discrete stochastic Gronwall lemma (cf. \cite{kruse2018}) yields for any $p \in (0,1)$
\begin{equation*}
\mathds{E}\left[\sup_{t \in [0,T]\wedge \xi}Z_{\nu(t)}^{p}\right] \leq \mathcal{C}
\end{equation*}
where (under smallness assumption on $h$) $\mathcal{C} = \mathcal{C}(p,T,M,Q)>0$ is independent of $h$.\\
For the second claim, using \eqref{YXref} we estimate
\begin{equation*}
\begin{aligned}
\sum_{i=1}^M\left\|X_{\nu_{+}(s)}^{(i),a}\right\|^2 &\leq \sum_{i=1}^M \left\|X_{\nu_{+}(s)}^{(i),f}\right\|^2\\
&\hspace{1.25cm} +h\left(\left\|X_{\nu_{+}(s)}^{(i),f}\right\|^2 +\left\|K_{\nu_{+}(s)}\right\|^2\|g\|_{\text{Lip}}^2\left\|X_{\eta_{+}(s)}^{\text{ref}}-X_{\nu_{+}(s)}^{(i),f}\right\|^2\right)\\
&\hspace{1.25cm} +h\left\|X_{\nu_{+}(s)}^{(i),f}\right\|^2 + \left\|K_{\nu_{+}(s)}\right\|^2\|g\|_{\text{Lip}}^2\int_{\eta(s)}^{\eta_{+}(s)}\left\|X_r^{\text{ref}}-X_{\eta_{+}(s)}^{\text{ref}}\right\|^2{\rm d}r\\
&\hspace{1.25cm} + 2\left \langle X_{\nu_{+}(s)}^{(i),f},K_{\nu_{+}(s)}C^{\frac{1}{2}}\int_{\eta(s)}^{\eta_{+}(s)}{\rm d}\left(V_r + V_r^{(i)}\right)\right \rangle\\
&\hspace{1.25cm} + 4h^2\left\|K_{\nu_{+}(s)}\right\|^2\|g\|_{\text{Lip}}^2\left\|X_{\eta_{+}(s)}^{\text{ref}}-X_{\nu_{+}(s)}^{(i),f}\right\|^2\\
&\hspace{1.25cm} + 4h\left\|K_{\nu_{+}(s)}\right\|^2\|g\|_{\text{Lip}}^2\int_{\eta(s)}^{\eta_{+}(s)}\left\|X_r^{\text{ref}}-X_{\eta_{+}(s)}^{\text{ref}}\right\|^2{\rm d}r\\
&\hspace{1.25cm} + 4\left\|\int_{\eta(s)}^{\eta_{+}(s)}K_{\nu_{+}(s)}C^{\frac{1}{2}}{\rm d}V_s\right\|^2 + 4\left\|K_{\nu_{+}(s)}C^{\frac{1}{2}}V_{\nu_{+}(s)}^{(i)}\right\|^2.
\end{aligned}
\end{equation*}
With the crude estimate
\[\left\|X_{\eta_{+}(s)}^{\text{ref}}-X_{\nu_{+}(s)}^{(i),f}\right\|^2 \leq 2\left(\left\|X_{\eta_{+}(s)}^{\text{ref}}\right\|^2+\left\|X_{\nu_{+}(s)}^{(i),f}\right\|^2\right)\]
as well as with
\begin{equation*}
\begin{aligned}
\left\|X_{\nu_{+}(s)}^{(i),f}\right\|^2 &\leq \left\|X_{\nu(s)}^{(i),a}\right\|^2 + 2hC_f\left(1+\left\|X_{\nu(s)}^{(i),a}\right\|^2\right)+2h^2\tilde{C}_f\left(1+\left\|X_{\nu(s)}^{(i),a}\right\|^2\right)\notag\\
&\quad + 2\left \langle X_{\nu(s)}^{(i),a},Q^{\frac{1}{2}}W_{\nu_{+}(s)}^{(i)}\right \rangle + 2\left\|Q^{\frac{1}{2}}W_{\nu_{+}(s)}^{(i)}\right\|^2
\end{aligned}
\end{equation*}
we get an estimate of the form
\begin{equation*}
\sum_{i=1}^M \left\|X_{\nu_{+}(s)}^{(i),a}\right\|^2 \leq \left(1+h\mathcal{C}\left(\left\|K_{\nu_{+}(s)}\right\|^2,h\right)\right)\left(\sum_{i=1}^M \left\|X_{\nu(s)}^{(i),a}\right\|^2\right) + R_{\nu_{+}(s)} + N_{\nu_{+}(s)}
\end{equation*}
where
\begin{equation*}
\begin{aligned}
1+h\mathcal{C}\left(\left\|K_{\nu_{+}(s)}\right\|^2,h\right) &:= \left(1+2hC_f+2h^2\tilde{C}_f\right)\\
&\quad \times\left(1+2h+2h(1+4h)\left\|K_{\nu_{+}(s)}\right\|^2\|g\|_{\text{Lip}}^2\right),
\end{aligned}
\end{equation*}
\begin{equation*}
\begin{aligned}
R_{\nu_{+}(s)} &:= \sum_{i=1}^M 2hC_f +2h^2\tilde{C}_f\\
&\hspace{1.5cm} +2h(1+4h)\left\|K_{\nu_{+}(s)}\right\|^2\|g\|_{\text{Lip}}^2\left\|X_{\eta_{+}(s)}^{\text{ref}}\right\|^2\\
&\hspace{1.5cm} + (1+4h)\left\|K_{\nu_{+}(s)}\right\|^2\|g\|_{\text{Lip}}^2\int_{\eta(s)}^{\eta_{+}(s)}\left\|X_r^{\text{ref}}-X_{\eta_{+}(s)}^{\text{ref}}\right\|^2{\rm d}r\\
&\hspace{1.5cm} + 2\left\|Q^{\frac{1}{2}}W_{\nu_{+}(s)}^{(i)}\right\|^2\notag\\
&\hspace{1.5cm} +4\left\|\int_{\eta(s)}^{\eta_{+}(s)}K_{\nu_{+}(s)}C^{\frac{1}{2}}{\rm d}V_r\right\|^2 + 4\left\|\int_{\eta(s)}^{\eta_{+}(s)}K_{\nu_{+}(s)}C^{\frac{1}{2}}{\rm d}V_r^{(i)}\right\|^2,
\end{aligned}
\end{equation*}
\begin{align*}
N_{\nu_{+}(s)} &:=2\sum_{i=1}^M \left \langle X_{\nu(s)}^{(i),a}, Q^{\frac{1}{2}}W_{\nu_{+}(s)}^{(i)}\right \rangle \\
&\hspace{1.5cm}+ \left \langle X_{\nu_{+}(s)}^{(i),f}, K_{\nu_{+}(s)}C^{\frac{1}{2}}\int_{\eta(s)}^{\eta_{+}(s)}K_{\nu_{x}(s)}C^{\frac{1}{2}}{\rm d}\left(V_r+V_r^{(i)}\right)\right \rangle.
\end{align*}
Iterating this gives
\begin{equation*}
\begin{aligned}
\sum_{i=1}^M \left\|X_{\nu(s)}^{(i),a}\right\|^2 &\leq \left(1+h\mathcal{C}\left(\left\|K_{\nu(s)}\right\|^2,h\right)\right)^{\nu(s)}\left(\sum_{i=1}^M\left\|X_0^{(i),a}\right\|^2\right)\\
&\quad +\sum_{k=1}^{\nu(s)}\left(1+h\mathcal{C}\left(\left\|K_{\nu_{+}(s)}\right\|^2,h\right)\right)^{\nu(s)-k}\left(R_k+N_k\right).
\end{aligned}
\end{equation*}
Note that for $s \leq \tau_n^{h}$ it holds $\left\|K_{\nu_(s)}\right\|^2 \lesssim n^2$ by Appendix \ref{appLemma}, thus for small enough $h$ we obtain
\[1+h\mathcal{C}\left(\left\|K_{\nu(s)}\right\|^2,h\right) \leq e^{T\mathcal{C}\left(n^2,h\right)}.\]
Therefore
\begin{equation*}
\begin{aligned}
&\mathds{E}\left[\int_0^{T\wedge \tau_n^{h}} \sum_{i=1}^M \left\|X_{\nu(s)}^{(i),a}\right\|^2{\rm d}s\right] \\
&\leq e^{T\mathcal{C}\left(n^2,h\right)}\mathds{E}\left[\int_0^{T\wedge\tau_n^{h}} \sum_{i=1}^M \left\|X_0^{(i),a}\right\|^2 + \sum_{k=1}^{\nu(s)}(R_k+N_k){\rm d}s\right]\\
&\leq Te^{T\mathcal{C}\left(n^2,h\right)}\mathds{E}\left[\sum_{i=1}^M\left\|X_0^{(i),a}\right\|^2\right]\\
&\quad + e^{T\mathcal{C}\left(n^2,h\right)}\left(\mathds{E}\left[\int_0^{T\wedge \tau_n^{h}}\sum_{k=1}^{\nu(s)}R_k{\rm d}s\right] + \mathds{E}\left[\int_0^{T\wedge \tau_n^{h}}\sum_{k=1}^{\nu(s)}N_k{\rm d}s\right]\right).
\end{aligned}
\end{equation*}
\noindent
First of all observe that
\begin{equation*}
\mathds{E}\left[\int_0^{T\wedge \tau_n^{h}}\sum_{k=1}^{\nu(s)}N_k{\rm d}s\right] \leq \int_0^T\sum_{k=1}^{\nu(s)}\mathds{E}\left[N_k\right]{\rm d}s = 0
\end{equation*}
since $N$ is a martingale. For the other remainder we estimate
\begin{equation*}
\begin{aligned}
&\mathds{E}\left[\int_0^{T\wedge \tau_n^{h}}\sum_{k=1}^{\nu(s)}R_k{\rm d}s\right]\\
&\lesssim \mathds{E}\left[\int_0^{T\wedge\tau_n^{h}} \sum_{k=1}^{\nu(s)}\sum_{i=1}^M 2hC_f + 2h^2\tilde{C}_f  + 2\left\|Q^{\frac{1}{2}}W_k^{(i)}\right\|^2\right.\\
&\hspace{3.5cm}\left. + 2h(1+4h)n^2\|g\|_{\text{Lip}}^2\left(\int_{t_{k-1}}^{t_k}\left\|X_r^{\text{ref}}-X_{t_k}^{\text{ref}}\right\|^2{\rm d}r\right) {\rm d}s\right]\\
&\quad + \mathds{E}\left[\int_0^{T\wedge\tau_n^{h}}\sum_{k=1}^{\nu(s)}M\left\|\int_{t_{k-1}}^{t_k}K_kC^{\frac{1}{2}}{\rm d}V_r\right\|^2{\rm d}s\right] \\
&\quad+ \mathds{E}\left[\int_0^{T\wedge\tau_n^{h}} \sum_{k=1}^{\nu(s)}\sum_{i=1}^M \left\|\int_{t_{k-1}}^{t_k}K_kC^{\frac{1}{2}}{\rm d}V_r^{(i)}\right\|^2{\rm d}s\right].
\end{aligned}
\end{equation*}
It holds by the Burkholder-Davis-Gundy inequality that
\begin{equation*}
\begin{aligned}
&\mathds{E}\left[\int_0^{T\wedge\tau_n^{h}}\sum_{k=1}^{\nu(s)}M\left\|\int_{t_{k-1}}^{t_k}K_kC^{\frac{1}{2}}{\rm d}V_r\right\|^2{\rm d}s\right]\\
&\leq C_{BDG}MTL\mathds{E}\left[\sup_{t \leq T\wedge \tau_n^{h}}\int_{\eta(t)}^{\eta_{+}(t)} \text{tr}\left(K_{\nu_{+}(t)}CK_{\nu_{+}(t)}^T\right){\rm d}r\right] \\
&\lesssim  C_{BDG}MT^2n^2\text{tr}(C).
\end{aligned}
\end{equation*}
Note that by assuming that
\begin{equation*}
\sup_{t \in [0,T]} \mathds{E}\left[\left\|X_t^{\text{ref}}\right\|^2\right] < \infty
\end{equation*}
one can show that
\begin{equation*}
\mathds{E}\left[\left\|X_t^{\text{ref}} - X_s^{\text{ref}}\right\|^2\right] \leq 2h^2C_f\left(1 + \sup_{r \in [0,T]}\mathds{E}\left[\left\|X_r^{\text{ref}}\right\|^2\right]\right) + 2\left\|Q^{\frac{1}{2}}\right\|^2h.
\end{equation*}
With this we obtain in total that
\[\mathds{E}\left[\int_0^{T\wedge \tau_n^{h}}\sum_{k=1}^{\nu(s)}R_k{\rm d}s\right]\]
is bounded, which concludes the proof.

\section{Accuracy results}

\subsection{Proof of Theorem \ref{AccCont}}\label{appAccCont}
The procedure of the proof is closely related to a comparable result in \cite{deWiljes2018}: let
\[ e_t := \frac{1}{2}\sum_{i=1}^M \left\|X_t^{\text{ref}} - X_t^{(i)}\right\|^2\]
denote the mean-squared approximation error between $X^{\text{ref}}$ and $X^{(i)}$. Using \eqref{YXref} together with It\^{o}'s formula, $e$ satisfies the evolution equation
\begin{equation*}
{\rm d}e_t = \mathcal{E}_t {\rm d}t + {\rm d}\mathcal{M}_t
\end{equation*}
where in the case of the EnKBF we obtain
\begin{align*}
\mathcal{E}_t &:= \sum_{i=1}^M \left \langle X_t^{\text{ref}} - X_t^{(i)}, f\left(X_t^{\text{ref}}\right) - f\left(X_t^{(i)}\right)\right \rangle \notag\\
&\hspace{1.5cm}-  \left \langle X_t^{\text{ref}} - X_t^{(i)}, K_t\left(g\left(X_t^{\text{ref}}\right)-g\left(X_t^{(i)}\right)\right)\right \rangle \notag\\
&\hspace{1.5cm} + \left(\text{tr}(Q) + \text{tr}\left(K_tCK_t^T\right)\right),\\
{\rm d}\mathcal{M}_t&:= \sum_{i=1}^M \left\langle X_t^{\text{ref}} - X_t^{(i)}, Q^{\frac{1}{2}}{\rm d}\left(W_t - W_t^{(i)}\right)\right \rangle\notag\\
&\hspace{1.5cm} - \left \langle X_t^{\text{ref}} - X_t^{(i)}, K_tC^{\frac{1}{2}}\left({\rm d}V_t + {\rm d}V_t^{(i)}\right)\right \rangle
\end{align*}
and in case of the ETKBF
\begin{align*}
\mathcal{E}_t &:= \sum_{i=1}^M \left \langle X_t^{\text{ref}} - X_t^{(i)}, f\left(X_t^{\text{ref}}\right) - f\left(X_t^{(i)}\right)\right \rangle \notag\\
&\hspace{1.5cm}-  \left \langle X_t^{\text{ref}} - X_t^{(i)}, K_t\left(g\left(X_t^{\text{ref}}\right)-\frac{1}{2}\left(g\left(X_t^{(i)}\right)+\bar{g}_t\right)\right)\right \rangle \notag\\
&\hspace{1.5cm} + \left(\text{tr}(Q) + \frac{1}{2}\text{tr}\left(K_tCK_t^T\right)\right),\\
{\rm d}\mathcal{M}_t&:= \sum_{i=1}^M \left\langle X_t^{\text{ref}} - X_t^{(i)}, Q^{\frac{1}{2}}{\rm d}\left(W_t - W_t^{(i)}\right)\right \rangle\notag\\
&\hspace{1.5cm} - \left \langle X_t^{\text{ref}} - X_t^{(i)}, K_tC^{\frac{1}{2}}{\rm d}V_t\right \rangle.
\end{align*}
Using Lipschitz-continuity of $f$ and $g$ we obtain the estimate
\begin{equation*}
\mathcal{E}_t \leq 2\left( |f|_{+} + \|K_t\|(Lg)_{+}\right)e_t + \text{tr}(Q) + \|K_t\|^2\left\|C^{\frac{1}{2}}\right\|_F^2
\end{equation*}
for the EnKBF, and
\begin{equation*}
\mathcal{E}_t \leq \left( 1+ 2|f|_{+} + 2\|K_t\|(Lg)_{+}\right)e_t + \text{tr}(Q) + \|K_t\|^2\left\|C^{\frac{1}{2}}\right\|_F^2
\end{equation*}
for the ETKBF. On $[0, T \wedge \tau_n]$ using the estimate on Kalman gain from Appendix \ref{appLemma} yields in both cases
\begin{equation*}
\mathcal{E}_t \leq \alpha_n e_t + \beta_n,
\end{equation*}
thus using a Gronwall argument we obtain the estimate
\begin{equation*}
\begin{aligned}
&\mathds{E}\left[\sup_{t \in [0, T \wedge \tau_n]} \sum_{i=1}^M \left\|X_t^{\text{ref}} - X_t^{(i)}\right\|^2\right]\\
&\leq 2e^{\alpha_n T}\mathds{E}\left[\sum_{i=1}^M \left\|X_0^{\text{ref}}-X_0^{(i)}\right\|^2\right] + \frac{2\beta_n}{\alpha_n}\left(e^{\alpha_n T} - 1\right) + 2\mathds{E}\left[\sup_{t \in [0, T \wedge \tau_n]} \int_0^t e^{\alpha_n(t-s)} {\rm d}\mathcal{M}_s\right].
\end{aligned}
\end{equation*}
Using the Burkholder-Davis-Gundy inequality, the last summand can be estimated by
\begin{equation*}
\begin{aligned}
&\mathds{E}\left[\sup_{t \in [0, T \wedge \tau_n]} \int_0^t e^{\alpha_n(t-s)}{\rm d}\mathcal{M}_s\right] \leq C_{BDG} \mathds{E}\left[\left( \int_0^{T\wedge \tau_n} e^{2\alpha_n(T-s)}{\rm d}\langle \mathcal{M}\rangle_s\right)^{\frac{1}{2}}\right]\\
&\leq \frac{C_{BDG}}{2}\left(1 + \mathds{E}\left[ \int_0^{T\wedge \tau_n} e^{2\alpha_n(T-s)}{\rm d}\langle \mathcal{M}\rangle_s\right]\right)\\
&\leq \frac{C_{BDG}}{2}\left( 1+ \int_0^{T} B(n)e^{2\alpha_n(T-s)}\mathds{E}\left[\sup_{r \in [0,s\wedge \tau_n]}\sum_{i=1}^M\left\|X_r^{\text{ref}} - X_r^{(i)}\right\|^2\right]{\rm d}s\right)
\end{aligned}
\end{equation*}
where
\[ B(n) := (1+M)\left(\left\|Q\right\|+ \left(\kappa_K n^2\right)^2\|C\|\right)\]
in case of the EnKBF and
\[B(n) := (1+M)\|Q\|+M\left(\kappa_K n^2\right)^2\|C\|\]
in case of the ETKBF. Using a Gronwall argument then yields
\begin{equation}\label{RefContDetail}
\begin{aligned}
&\mathds{E}\left[\sup_{t \in [0, T\wedge \tau_n]} \sum_{i=1}^M \left\|X_t^{\text{ref}} - X_t^{(i)}\right\|^2\right]\\
&\leq \text{exp}\left(\frac{C_{BDG}B(n)}{2\alpha_n}\left(e^{2\alpha_nT}-1\right)\right)\\
&\hspace{0.5cm}\times \left( 2e^{\alpha_nT}\mathds{E}\left[\sum_{i=1}^M\left\|X_0^{\text{ref}} - X_0^{(i)}\right\|^2\right] + 2\frac{\beta_n}{\alpha_n}\left(e^{\alpha_n T}-1\right) + C_{BDG}\right).
\end{aligned}
\end{equation}

\subsection{Proof of Theorem \ref{AccDiscr}}\label{appAccDiscr}
Again for the sake of conciseness we only present the proof for the EnKF, whereas the proof for the ESRF algorithms is analogous and follows with the help of estimates in Appendix \ref{appLemma}. Using \eqref{YXref} we decompose
\begin{equation*}
\begin{aligned}
X_t^{\text{ref}}-X_{\nu(t)}^{(i),a} &= X_{\eta(t)}^{\text{ref}}-X_{\nu(t)}^{(i),a} + X_t^{\text{ref}}-X_{\eta(t)}^{\text{ref}}\\
&= X_0^{\text{ref}}-X_0^{(i),a}\\
&\quad + \int_0^{\eta(t)}f\left(X_s^{\text{ref}}\right)-f\left(X_{\nu(s)}^{(i),a}\right) - K_{\nu_{+}(s)}\left(g\left(X_s^{\text{ref}}\right)-g\left(X_{\nu_{+}(s)}^{(i),f}\right)\right){\rm d}s\\
&\quad - \int_0^{\eta(t)}K_{\nu_{+}(s)}C^{\frac{1}{2}}{\rm d}\left(V_s+V_s^{(i)}\right)\\
&\quad +Q^{\frac{1}{2}}\left(W_{\eta(t)}-W_{\eta(t)}^{(i)}\right)\notag\\
&\quad + X_t^{\text{ref}}-X_{\eta(t)}^{\text{ref}}
\end{aligned}
\end{equation*}
thus
\begin{equation*}
\begin{aligned}
&\sum_{i=1}^M \left\|X_t^{\text{ref}}-X_{\nu(t)}^{(i),a}\right\|^2\\ &\lesssim \sum_{i=1}^M \left\|X_0^{\text{ref}}-X_0^{(i),a}\right\|^2\\
&\quad + \eta(t) \int_0^{\eta(t)} \|f\|_{\text{Lip}}^2\sum_{i=1}^M\left\|X_s^{\text{ref}}-X_{\nu(s)}^{(i),a}\right\|^2 \\
&\hspace{3cm}+ \left\|K_{\nu_{+}(s)}\right\|^2\|g\|_{\text{Lip}}^2\sum_{i=1}^M \left\|X_s^{\text{ref}}-X_{\nu_{+}(s)}^{(i),f}\right\|^2{\rm d}s\\
&\quad + M\left\|\int_0^{\eta(t)} K_{\nu_{+}(s)}C^{\frac{1}{2}}{\rm d}V_s\right\|^2 + \sum_{i=1}^M\left\|\int_0^{\eta(t)}K_{\nu_{+}(s)}C^{\frac{1}{2}}{\rm d}V_s^{(i)}\right\|^2\\
&\quad + \sum_{i=1}^M\left\|Q^{\frac{1}{2}}\left(W_{\eta(t)}-W_{\eta(t)}^{(i)}\right)\right\|^2 + M\left\|X_t^{\text{ref}}-X_{\eta(t)}^{\text{ref}}\right\|^2.
\end{aligned}
\end{equation*}
Use the estimates
\begin{align*}
&\sum_{i=1}^M\left\|X_s^{\text{ref}}-X_{\nu_{+}(s)}^{(i),f}\right\|^2 \\
&\lesssim \sum_{i=1}^M \left\|X_s^{\text{ref}}-X_{\nu(s)}^{(i),a}\right\|^2 + h^2\tilde{C}_f\left(M+\sum_{i=1}^M\left\|X_{\nu(s)}^{(i),a}\right\|^2\right)+\sum_{i=1}^M\left\|Q^{\frac{1}{2}}W_{\nu_{+}(s)}^{(i)}\right\|^2
\end{align*}
as well as
\begin{equation*}
\left\|X_t^{\text{ref}}-X_{\eta(t)}^{\text{ref}}\right\|^2 \leq 2(t-\eta(t))\int_{\eta(t)}^t\tilde{C}_f\left(1+\left\|X_r^{\text{ref}}\right\|^2\right){\rm d}r + 2\|Q\|\left\|W_t-W_{\eta(t)}\right\|^2,
\end{equation*}
then
\begin{equation*}
\begin{aligned}
&\mathds{E}\left[\sup_{t \in [0, T \wedge \tau_n^{h}]}\sum_{i=1}^M\left\|X_t^{\text{ref}}-X_{\nu(t)}^{(i),a}\right\|^2\right]\\
&\lesssim \mathds{E}\left[\sum_{i=1}^M \left\|X_0^{\text{ref}}-X_0^{(i),a}\right\|^2\right]\\
&\quad + T\mathds{E}\left[\int_0^{T \wedge \tau_n^{h}} \left(\|f\|_{\text{Lip}}^2+n^2\|g\|_{\text{Lip}}^2\right)\sup_{r \in[0, s\wedge \tau_n^{h}]}\sum_{i=1}^M\left\|X_r^{\text{ref}}-X_{\nu(r)}^{(i),a}\right\|^2\right.\\
&\hspace{2.5cm}\left. + n^2\|g\|_{\text{Lip}}^2\left(h^2\tilde{C}_f\left(M+\sum_{i=1}^M\left\|X_{\nu(s)}^{(i),a}\right\|^2\right)+\sum_{i=1}^M\left\|Q^{\frac{1}{2}}W_{\nu_{+}(s)}^{(i)}\right\|^2\right){\rm d}s\right]\\
&\quad + M\mathds{E}\left[\sup_{t\in [0, T \wedge \tau_n^{h}]}\left\|\int_0^{\eta(t)} K_{\nu_{+}(s)}C^{\frac{1}{2}}{\rm d}V_s\right\|^2\right] \\
&\quad+ \sum_{i=1}^M \mathds{E}\left[\sup_{t \in [0, T \wedge \tau_n^{h}]} \left\|\int_0^{\eta(t)} K_{\nu_{+}(s)}C^{\frac{1}{2}}{\rm d}V_s^{(i)}\right\|^2\right]\\
&\quad + \sum_{i=1}^M \mathds{E}\left[\sup_{t \in [0, T \wedge \tau_n^{h}]}\left\|Q^{\frac{1}{2}}\left(W_{\eta(t)}-W_{\eta(t)}^{(i)}\right)\right\|^2\right]\\
&\quad + 2M\mathds{E}\left[ \sup_{t \in [0, T \wedge \tau_n^{h}]} \left(h\int_{\eta(t)}^t\tilde{C}_f\left(1+\left\|X_r^{\text{ref}}\right\|^2\right){\rm d}r + \|Q\|\left\|W_t-W_{\eta(t)}\right\|^2 \right)\right].
\end{aligned}
\end{equation*}
Using Theorem \ref{WellposedEnKFs} as well as that
\begin{equation*}
\mathds{E}\left[\sup_{t\in [0, T \wedge \tau_n^{h}]}\left\|\int_0^{\eta(t)} K_{\nu_{+}(s)}C^{\frac{1}{2}}{\rm d}V_s\right\|^2\right] \lesssim Tn^2\text{tr}(C),
\end{equation*}
we obtain an estimate of the form
\begin{equation*}
\begin{aligned}
&\mathds{E}\left[\sup_{t \in [0, T \wedge \tau_n^{h}]}\sum_{i=1}^M\left\|X_t^{\text{ref}}-X_{\nu(t)}^{(i),a}\right\|^2\right] \\
&\lesssim \mathds{E}\left[\sum_{i=1}^M\left\|X_0^{\text{ref}}-X_0^{(i),a}\right\|^2\right] + L(n)\int_0^T\mathds{E}\left[\sup_{r \in[0, s \wedge \tau_n^{h}]}\sum_{i=1}^M\left\|X_r^{\text{ref}}-X_{\nu(r)}^{(i),a}\right\|^2\right]{\rm d}s \\
&\quad+ R(n,h)
\end{aligned}
\end{equation*}
which by a Gronwall argument yields the claim.
\end{document}